\RequirePackage{etoolbox}
\csdef{input@path}{{style/}{graphics/}}
\documentclass[aap,MSNbibl,seceqn,dvips]{arximspdf}
\usepackage{mathbh}
\usepackage{graphicx}

\doi{10.1214/14-AAP1014} 
\volume{25}
\issue{2}
\pubyear{2015}
\firstpage{930}
\lastpage{973}

\makeatletter
\def\mid{|}
\newproclaim{ass}{Assumption}
\newcommand{\eqref}[1]{(\ref{#1})}
\def\emptyset{\varnothing}

\def\d{\delta}
\def\e{\varepsilon}

\def\L{\Lambda}

\newcommand{\R}     {\mathbb{R}}
\newcommand{\Z}     {\mathbb{Z}}
\newcommand{\N}     {\mathbb{N}}
\renewcommand{\P}   {\mathbb{P}}

\def\1{\mathbh{1}}
\newcommand{\ssup}[1]{{({#1})}}
\newtheorem{theorem}{Theorem}[section]
\newtheorem{lemma}[theorem]{Lemma}
\newtheorem{prop}[theorem] {Proposition}
\newtheorem{cor}[theorem] {Corollary}
\newproclaim{remark}[theorem] {Remark}
\newproclaim{defn}[theorem] {Definition}
\newproclaim{conj}[theorem] {Conjecture}
\newproclaim{step}{STEP}

\newproclaim{example}[theorem] {Example}
\newproclaim{bem}[theorem] {Remark}
\newproclaim{assumption}[theorem]{Assumption}

\renewcommand{\d}{\mathrm{d}}

\newcommand{\eps}{\varepsilon}

\newcommand{\Sym}{\mathfrak{S}}

\newcommand{\supp}{\operatorname{supp}}

\newcommand{\const}{\operatorname{cst.}}

\newcommand{\Ccal}   {{\mathcal C }}
\newcommand{\Dcal}   {{\mathcal D }}

\newcommand{\Mcal}   {{\mathcal M }}
\newcommand{\Ncal}   {{\mathcal N }}
\newcommand{\Ocal}   {{\mathcal O }}

\renewcommand{\e}   {\mathrm{e}}

\newcommand{\dd}{\mathrm{d}} 
\newcommand{\vect}[1]{\mathbf{#1}} 
\newcommand{\vectt}[1]{\bolds{#1}}

\newcommand{\cl}{\mathrm{cl}} 
\newcommand{\ideal}{\mathrm{ideal}} 
\newcommand{\Int}{\operatorname{Int}}
\newcommand{\dom}{\operatorname{dom}}

\makeatother

\begin{document}
\begin{frontmatter}

\title{Large deviations for cluster size distributions in a~continuous classical many-body system\thanksref{T1}}
\runtitle{Large deviations for cluster size distributions}
\thankstext{T1}{Supported by the
DFG-Forscher\-gruppe~718 ``Analysis and Stochastics in Complex
Physical Systems.''}

\begin{aug}
\author[A]{\fnms{Sabine}~\snm{Jansen}\ead[label=e1]{sabine.jansen@ruhr-uni-bochum.de}},
\author[B]{\fnms{Wolfgang}~\snm{K{\"o}nig}\corref{}\ead[label=e2]{koenig@wias-berlin.de}}
\and
\author[C]{\fnms{Bernd}~\snm{Metzger}}
\runauthor{S. Jansen, W. K{\"o}nig and B. Metzger}
\affiliation{Ruhr-Universit\"at Bochum, WIAS Berlin and TU Berlin, and WIAS Berlin}
\address[A]{S. Jansen\\
Faculty for Mathematics\\
Ruhr-Universit\"at Bochum\\
Universit\"atsstr. 150\\
44780 Bochum\\
Germany\\
\printead{e1}} 
\address[B]{W. K{\"o}nig\\
Weierstrass Institute Berlin\\
Mohrenstr.~39\\
10117 Berlin\\
Germany\\
and\\
Institute for Mathematics\\
Technische Universit\"at Berlin\\
Str.~des 17.~Juni 136\\
10623 Berlin\\
Germany\\
\printead{e2}}
\address[C]{B. Metzger\\
Weierstrass Institute Berlin\\
Mohrenstr.~39\\
10117 Berlin\\
Germany}
\end{aug}

\received{\smonth{6} \syear{2013}}
\revised{\smonth{1} \syear{2014}}

%
\begin{abstract}
An interesting problem in statistical physics is the condensation of
classical particles in droplets or clusters when the
pair-interaction is given by a stable Lennard--Jones-type potential. We
study two aspects of this problem. We start by deriving a large
deviations principle for
the cluster size distribution for any inverse temperature $\beta\in
(0,\infty)$ and
particle density $\rho\in(0,\rho_{\mathrm{cp}})$ in the thermodynamic
limit. Here $\rho_{\mathrm{cp}} >0$ is the close packing density. While in
general the rate function is an abstract object, our second main result
is the $\Gamma$-convergence of the rate function toward an explicit
limiting rate function in the low-temperature dilute limit $\beta\to
\infty$, $\rho\downarrow0$ such that $-\beta^{-1}\log\rho\to\nu$
for some $\nu\in(0,\infty)$. The limiting rate function and its
minimisers appeared in recent work,
where the temperature and the particle density were coupled with the
particle number.
In the decoupled limit considered here, we prove that just one cluster
size is dominant,
depending on the parameter $\nu$. Under additional assumptions on the potential,
the $\Gamma$-convergence along curves can be strengthened to uniform
bounds, valid in a low-temperature,
low-density rectangle.
\end{abstract}

%
\end{frontmatter}

\section{Introduction}\label{sec-Intro}
We consider interacting $N$-particle systems in a box $\Lambda
=[0,L]^d \subset\R^d$ with interaction energy
%
%
\begin{equation}
\label{Udef} U_N(x_1,\ldots, x_N):= \sum
_{1\leq i<j\leq N} v\bigl(|x_i-x_j|\bigr),
\end{equation}
where $v\dvtx[0,\infty) \to\R\cup\{\infty\}$ is a pair potential of
Lennard--Jones type; see Figure~\ref{lennard-jones-potential}. That is:
\begin{itemize}
\item it is large close to zero, inducing a repulsion that prevents the
particles from clumping;
\item it has a nondegenerate negative part, inducing an attraction,
that is, particles try to assume a certain fixed distance to each other;
\item it vanishes at infinity; that is, long-range effects are absent.
\end{itemize}
Additionally, we always assume that~$v$ is stable and has compact
support. We allow for the possibility that $v=\infty$ in some interval
$[0,r_{\mathrm{ hc}}]$ to represent hard core interaction. See Assumption \ref{assV}
in Section~\ref{sec-results} below for details.

%
\begin{figure}

\includegraphics{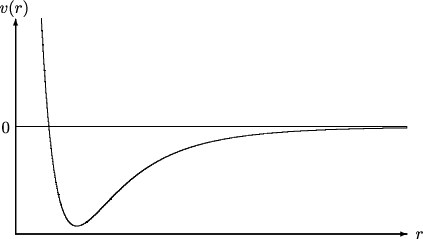}

\caption{The pair potential $v(r)=1.5 r^{-12}-5r^{-6}$ of
Lennard--Jones type.}\label{lennard-jones-potential}
\end{figure}

%
\begin{figure}[b]

\includegraphics{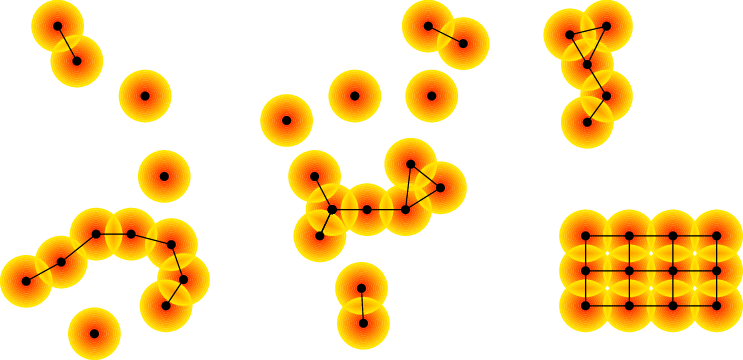}

\caption{A schematic figure illustrating the cluster decomposition of
a particle configuration and the induced graph structure.}\label{figure-2}
\end{figure}

A particle configuration $\vect{x}=(x_1,\ldots,x_N)$ in the box is
randomly structured into a number of smaller
subconfigurations, that is, well-separated smaller groups, which we
call \textit{clusters}; see Figure \ref{figure-2}.
One of our main questions is about the joint distribution of the
cluster sizes, that is, their cardinalities.
Intuitively, if the box size is large in comparison to the particle
number, then one expects many small clusters,
and if it is small, then one expects few large ones. We will analyse
this question more closely in the
thermodynamic limit, that is, keeping $\beta\in(0,\infty)$ fixed and taking
%
%
\begin{equation}
\label{thermolimit} N\to\infty,  L=L_N\to\infty \mbox{ such that}\qquad
\frac{N}{L_N^d}\to\rho,
\end{equation}
for some fixed \textit{particle density} $\rho\in(0,\infty)$, followed by
the dilute low-tempe\-ra\-tu\-re limit
%
%
\begin{equation}
\label{dilutelimit} \beta\to\infty, \rho\downarrow0 \mbox{ such that}\qquad{-}
\frac{1}\beta \log\rho\to\nu,
\end{equation}
for some $\nu\in(0,\infty)$. In this regime, the total entropy of the
system is well approximated by the sum of the
entropies of the clusters, and the excluded-volume effect between the
clusters as well as
the mixing entropy may be neglected. As a consequence, particles tend
to favor one optimal
cluster size, which depends on $\nu$ and may be infinite.

In recent work \cite{CKMS10}, the free energy was analysed in the
\textit{coupled} dilute low-temperature limit
%
%
\begin{eqnarray}
\label{Ndilutelimit} &&N\to\infty, \beta=\beta_N\to\infty, L=L_N
\to\infty \mbox{ such that}
\nonumber
\\[-8pt]
\\[-8pt]
\nonumber
&&\qquad-\frac{1}{\beta_N}\log\frac{N}{L_N^d}\to\nu,
\end{eqnarray}
with some constant $\nu\in(0,\infty)$. It was found that the limiting
free energy is a piecewise
linear, continuous function of $\nu$ with at least one kink, that is,
nondifferentiable point. Furthermore,
there was a phenomenological discussion of the interplay between the
limiting cluster distribution and
the kinks in the limiting free energy, on base of a variational
representation. See Section~\ref{sec-dilutelimit}
for details.

In the present paper, we go beyond \cite{CKMS10} by considering the
physically relevant
setting of a thermodynamic limit and by proving limit laws for the
quantities of interest. That
is, our two main purposes are:
\begin{longlist}[(ii)]
\item[(i)] to derive, for fixed $\beta, \rho\in(0,\infty)$, a large deviations
principle for the cluster size distribution in the thermodynamic limit
in \eqref{thermolimit}, and
\item[(ii)] to derive afterwards limit laws (laws of large numbers) for
the cluster size distribution in the
low-temperature dilute limit in \eqref{dilutelimit}.
\end{longlist}
In this way, we decouple the limit in
\eqref{Ndilutelimit} into taking two separate limits, and we prove
limit laws for the cluster
sizes in this regime.

The organisation of Section~\ref{sec-Intro} is as follows. In
Section~\ref{sec-model} we introduce our model and define the
thermodynamic set-up. Our main result concerning the large deviations
principle for the cluster size distribution is formulated in
Section~\ref{sec-results}. The low-temperature dilute limit is
discussed in Sections~\ref{sec-dilutelimit} and \ref{sec-Limitlaws}.
Adopting additional, stronger assumptions we give in Section~\ref
{sec-Uniform} bounds that are uniform in the temperature for dilute
systems. Finally we discuss in Section~\ref{sec-physics} some
mathematical and physical problems related to our results.

\subsection{The model and its thermodynamic set-up}\label{sec-model}
Here are our assumptions on the pair interaction potential that will be
in force throughout the paper.

{\renewcommand{\theass}{(V)}
\begin{ass}\label{assV}
The function $v\dvtx[0,\infty)
\to\R\cup\{\infty\}$ satisfies the following:
\begin{longlist}[(1)]
\item[(1)] $v$ is finite except possibly for a hard core: there is a
$r_\mathrm{hc}\geq0$
such that $v\equiv\infty$ on $(0,r_\mathrm{hc})$ and $v<\infty$ on
$(r_\mathrm{hc},\infty)$.
\item[(2)] $v$ is stable, that is, $ U_N(\vect{x})/N$ is bounded from
below in $N\in\N$ and $\vect{x}\in(\R^d)^N$.
\item[(3)] The support of $v$ is compact, more precisely, $b:=\sup
\supp(v)$ is finite.
\item[(4)] $v$ has an attractive tail: there is a $\delta\in(0,b)$
such that $v(r)<0$ for all
$r \in(b-\delta,b)$.
\item[(5)] $v$ is continuous in $[r_\mathrm{hc},\infty)$.
\end{longlist}
\end{ass}}

%
\begin{figure}[b]

\includegraphics{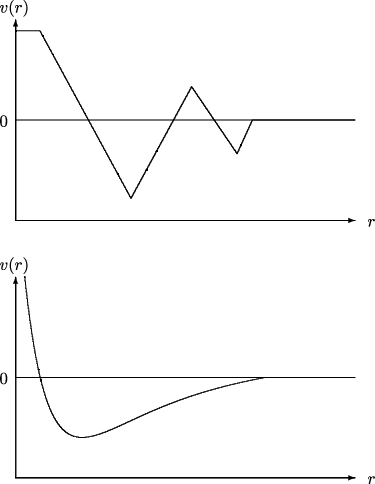}

\caption{Two examples of pair interaction potentials satisfying
Assumption \protect\ref{assV}.}
\label{Figure1}
\end{figure}

See Figure \ref{Figure1} for two examples.
Assumption \ref{assV} differs from Assumption (V) in \cite{CKMS10} in two
points: here we drop the requirement $v(r_\mathrm{hc})=\infty$, and
stability was there a consequence of some cumbersome additional assumption.

We introduce the Gibbs measure induced by the energy defined in \eqref
{Udef}. For $\beta\in(0,\infty)$, $N\in\N$ and a box $\Lambda\subset\R
^d$, we define the probability measure $\P_{\beta,\Lambda}^\ssup{N}$ on
$\Lambda^N$ by the Lebesgue density
%
%
\begin{equation}
\label{Pdef} \P_{\beta,\Lambda}^\ssup{N}(\d\vect{x})=\frac{1}{Z_\Lambda(\beta,N)N!}
\e^{ - \beta U_N(\vect{x})} \,\dd\vect{x},\qquad \vect{x}\in\Lambda^N,
\end{equation}
where
\[
Z_\Lambda(\beta,N):= \frac{1}{N!} \int_{\Lambda^N}
\e^{-\beta U_N(\vect
{x})} \,\dd\vect{x}
\]
is the canonical partition function at inverse temperature $\beta$.

We introduce the notions of connectedness and clusters. Fix $R\in
(b,\infty)$. Given $\vect{x}=(x_1,\ldots,x_N)\in(\R^d)^N$, we introduce
on the set $\{x_1,\ldots,x_N\}$ a graph structure by connecting two
points if their distance is $\leq \!R$. In this way, the notion of
$R$-connectedness is naturally introduced, which we also call just
connectedness. The connected components are also called \textit{clusters}.
A cluster of cardinality $k\in\N$ is called a $k$-\textit{cluster}. By
$N_k(\vect{x})$ we denote the number of $k$-clusters in $\vect{x}$, and by
\[
\rho_{k,\Lambda}(\vect{x}):= \frac{N_{k}(\vect{x})}{|\Lambda|}
\]
the $k$-\textit{cluster density}, the number of $k$-clusters per unit
volume. We consider the \textit{cluster size distribution}
%
%
\begin{equation}
\label{rhodef} \vectt{\rho}_\Lambda:= ( \rho_{k,\Lambda}
)_{k\in\N}
\end{equation}
as an $M_{N/|\Lambda|}$-valued random variable, where
%
%
\begin{equation}
\label{Mrhodef} M_{\rho}:= \biggl\{ (\rho_k)_{k\in\N}
\in[0,\infty)^\N \Big| \sum_{k\in\N} k
\rho_k \leq\rho \biggr\}, \qquad \rho\in(0,\infty).
\end{equation}
On $M_\rho$ we consider the topology of pointwise convergence, in which
it is compact. Note that for each finite $N$ and any box $\Lambda
\subset\R^d$,
\[
\sum_{k=1}^N k \rho_{k,\Lambda} (
\vect{x}) = \frac{N}{|\Lambda|},\qquad \vect{x}\in\Lambda^N.
\]
However, some mass of $\vectt{\rho}_\Lambda$ may be lost in the limit
$N\to\infty$ to infinitely large clusters. The distribution of $\vectt
{\rho}_\Lambda$ under the Gibbs measure $\P_{\beta,\Lambda}^\ssup{N}$
is the main object of our study.



Introduce the free energy per unit volume as
\[
f_\Lambda\biggl(\beta,{\frac{N}{|\Lambda|}} \biggr):= - \frac{1}{\beta|\Lambda
|}
\log Z_\Lambda(\beta,N).
\]
It is known \cite{Ru99} that the free energy per unit volume in the
thermodynamic limit,
%
%
\begin{equation}
\label{freeenergy} f(\beta,\rho):= \mathop{\lim_{N,L \to\infty}}
_{N/L^d \to\rho} f_{[0,L]^d}
\biggl(\beta,{\frac{N}{L^d}}\biggr),
\end{equation}
exists in $\R$ for all $\rho>0$ when there is no hard core, that is, if
$r_\mathrm{hc} =0$. When $r_\mathrm{hc} >0$, there is a threshold $\rho
_\mathrm{cp}>0$, the \textit{close packing density}, such that the limit
exists and is finite for $\rho\in(0,\rho_\mathrm{cp})$, and is $\infty
$ for
$\rho>\rho_\mathrm{cp}$. Since we are interested in dilute systems,
that is, small $\rho$, we will always
assume that $\rho\in(0,\rho_\mathrm{cp})$.

\subsection{Large deviations for cluster distribution under the Gibbs measure}\label{sec-results}
Our first main result is a large deviations principle (LDP)
for the cluster size distribution under the Gibbs measure. For the
concept of large deviations principles, see the monograph \cite{DZ98}.

%
\begin{theorem}[(Large deviation principle with convex rate
function)]\label{thm:ldp}
Fix $\beta\in(0,\infty)$ and $\rho\in(0,\rho_\mathrm{cp})$. Then,
in the thermodynamic limit $N\to\infty$, $L\to\infty$, $N/L^d \to
\rho$, the distribution of $\vectt{\rho}_\Lambda$ under $\P_{\beta,\Lambda}^\ssup{N}$ with $\Lambda=[0,L]^d$
satisfies a large deviations principle on $M_{\rho+\eps}$ with speed
$|\Lambda|=L^d$, where $\eps>0$ is such that $N/L^d \leq\rho+\eps$.
The rate function
$J_{\beta,\rho}\dvtx M_{\rho+\eps} \to[0,\infty]$ is convex, and its
effective domain
$\{ J_{\beta,\rho}(\cdot) <\infty\}$ is contained in $M_\rho$. For
$\rho$ sufficiently small, $\{ J_{\beta,\rho}(\cdot) <\infty\}$ is
equal to $M_\rho$.
\end{theorem}

If we impose $N/L^d \leq\rho$, the theorem also holds with $M_{\rho}$
instead of $M_{\rho+\eps}$.

The proof of Theorem~\ref{thm:ldp} is in Section~\ref{sec-ProofLDP}.
Define $f(\beta,\rho,\cdot)\dvtx M_\rho\to[0,\infty]$ through the equality
%
%
\begin{equation}
\label{eq:rf-fe} J_{\beta,\rho} (\vectt{\rho} )=: \beta \bigl( f(\beta,\rho,
\vectt{\rho}) - f(\beta,\rho) \bigr).
\end{equation}
Then the LDP may be rewritten, formally, as
\[
\frac{1}{N!} \int_{\Lambda^N} \e^{-\beta U_N(\vect{x})} \1 \bigl\{
\vectt{\rho}_\Lambda(\vect{x}) \approx\vectt{\rho} \bigr\} \,\dd\vect{x}
\approx\exp \bigl( - \beta|\Lambda| f(\beta,\rho,\vectt{\rho}) \bigr).
\]
Thus $f(\beta,\rho,\vectt{\rho})$ may be considered as the free energy associated
with the cluster size distribution $\vectt{\rho}_\Lambda$, thought of as
an order parameter.
The identity $\inf J_{\beta,\rho} =0$ translates into
\[
f(\beta,\rho) = \inf_{M_\rho} f(\beta,\rho,\cdot).
\]
In words: the (unconstrained) free energy is recovered as infimum of
the constrained free energy as the order parameter is varied, a
relation in the spirit of Landau theory.

It is a general fact from large deviations theory that an LDP implies
tightness. More specifically, the LDP of Theorem~\ref{thm:ldp} implies
a limit law for the cluster size distribution toward the set of
minimisers of the rate function. This is even a law of large numbers if
this set is a singleton. Hence, Theorem~\ref{thm:ldp} gives us control
on the limiting behaviour of the cluster size distribution under the
Gibbs measure in the thermodynamic limit. However, in the general
context of Theorem~\ref{thm:ldp}, we cannot offer any formula for the
rate function $J_{\beta,\rho}$. We have to restrict ourselves to the
low-temperature dilute limit \eqref{dilutelimit}. In this setting we
obtain explicit asymptotic formulae in Section~\ref{sec-dilutelimit}
below, and this is our second main result.

\subsection{The dilute low-temperature limit of the rate function}\label
{sec-dilutelimit}
In this section, we formulate and comment on our main result about the
limiting behaviour of the LDP rate function $J_{\beta,\rho}$ introduced
in Theorem~\ref{thm:ldp} and of its minimisers in the dilute
low-temperature limit in \eqref{dilutelimit}. This behaviour is
explicitly identified in terms of the \textit{ground-state energy} of $U_N$,
\[
E_N:= \inf_{\vect{x}\in(\R^d)^N} U_N(\vect{x}),\qquad N\in
\N.
\]
It can be seen as  in the proof of \cite{CKMS10}, Lemma~1.1, using
subadditivity that the limit
\[
e_\infty:= \lim_{N\to\infty} \frac{E_N}N \in(-
\infty,0)
\]
exists. It lies in the nature of the regime in \eqref{dilutelimit} that
it is not the cluster size distribution $\rho_k$ that will converge
toward an interesting limit (actually, these will vanish), but the
term $q_k=k\rho_k/\rho$, which carries the interpretation of frequency
of particles in $k$-clusters. Therefore, let
\[
\mathcal{Q}:= \biggl\{ \vect{q} = (q_k)_{k\in\N}
\in[0,1]^\N \Big| \sum_{k\in\N} q_k
\leq1 \biggr\}
\]
and introduce, for $\nu\in(0,\infty)$, the map $g_\nu\dvtx\mathcal
{Q} \to\R$ defined by
%
%
\begin{equation}
\label{gnudef} g_\nu (\vect{q} ):= \sum_{k\in\N}
q_k \frac{E_k-\nu}{k}+ \biggl( 1- \sum_{k\in\N}q_k
\biggr)e_\infty.
\end{equation}
Our second main result is the following.

%
\begin{theorem}[($\Gamma$-convergence of the rate function)]\label{thm:main}
Let $\nu\in(0,\infty)$. In the limit $\beta\to\infty$, $\rho\to0$
such that $-\beta^{-1} \log\rho\to\nu$, the function
\[
\mathcal{Q} \to\R\cup\{\infty\},\qquad \vect{q}= (q_k)_{k\in\N}
\mapsto\frac{1}\rho f \biggl(\beta,\rho, \biggl(\frac{\rho q_k}{k}
\biggr)_{k\in\N} \biggr)
\]
$\Gamma$-converges to $g_\nu$.
\end{theorem}

For the notion of $\Gamma$-convergence, see the monograph \cite{dMaso}.
Theorem~\ref{thm:main} is proved in Section~\ref{sec-proofThm14}. The
physical intuition is the following: at low density, the particle
system can be approximated by an \emph{ideal gas of clusters}; see~\cite
{hillbook},
Chapter~5 or~\cite{sator}. ``Ideal'' means that we
neglect the ``excluded volume,'' that is,
the constraint that clusters have mutual distance $\geq R$. As can be
seen from the proof of Lemma~\ref{prop:lb},
this means that the rate function $f(\beta,\rho,\cdot)$ is well
approximated by the ideal free energy
%
%
\begin{eqnarray}
\label{eq:fideal-def} f^\ideal\bigl(\beta,\rho,(\rho_k)_k
\bigr)&:=& \sum_{k\in\N} k \rho_k
f_k^\cl (\beta) + \biggl(\rho- \sum
_{k\in\N} k \rho_k \biggr) f_\infty^\cl(
\beta)
\nonumber
\\[-8pt]
\\[-8pt]
\nonumber
&&{}+ \frac{1}\beta\sum_{k\in\N}
\rho_k (\log\rho_k - 1).
\end{eqnarray}
Here $f_k^\cl(\beta)$ and $f_\infty^\cl(\beta)$ should be thought of as
free energies per particle
in clusters of size $k$ (resp., in infinitely large clusters); see
Section~\ref{sec-boundsf}
for the precise definitions. The functional $\rho g_\nu$ is obtained
from $f^\ideal$ by two simplifications, justified at low temperatures.
\begin{itemize}
\item First, we approximate cluster internal free energies by their
ground state energies.
\item Second, we split the entropic term as
\[
\frac{1}{\beta} \sum_{k \in\N} \rho_k (
\log\rho_k - 1) = \sum_{k\in\N
}
\rho_k \frac{\log\rho}{\beta} + \frac{1}\beta\sum
_{k \in\N} \rho_k \biggl(\log\frac{\rho_k}{\rho} - 1
\biggr)
\]
and keep only the first sum. Thus we keep the entropic contribution
coming from the ways to place
the clusters (their centers of gravity) in the box and discard the
mixing entropy.
\end{itemize}

Since these two simplifications suppress physical intuition to some
extent, it appears natural to further analyse the consequences of the
approximation with the above ideal free energy; this is carried out in
\cite{JK}. Interesting connections with well-known cluster expansions
are discussed in \cite{J}.

In classical statistical physics, the approach we take here goes under
the name of a geometric, or droplet,
picture of condensation~\cite{hillbook,sator}. This is closely related
to the well-known contour picture
of the Ising model and lattice gases~\cite{Ru99}. Lattice gas cluster
sizes have been studied, for example,
in~\cite{lebowitz}, continuous systems were investigated in~\cite
{murmann,zessin}. The focus of these works
was on parameter regions where only small clusters occur. Our declared
goal, in contrast, is
to derive bounds that cover both the small cluster and the large
cluster regimes (in the notation introduced
below, this means both $\nu> \nu^*$ and $\nu< \nu^*$).

Under additional assumptions on the pair potential, we can replace the
somewhat abstract $\Gamma$-convergence result with more concrete
uniform error bounds; see equations~\eqref{estiratefct} to~\eqref
{eq:largenu} in Theorem~\ref{thm:unif}.

The rate function $g_\nu$ appeared in \cite{CKMS10} in the description
of the behaviour of the partition function $Z_{\beta,\Lambda}^\ssup{N}$
in the coupled dilute low-temperature limit in \eqref{Ndilutelimit}.
More precisely, it was shown there that, in this limit, for any $\nu\in
(0,\infty)$,
\[
-\frac{1}{N\beta_N}\log Z_{\beta_N,\Lambda_N}^\ssup{N}\to\mu(\nu).
\]
It was phenemenologically discussed, but it was not given mathematical
substance to, the conjecture that the random variable $\vect{q}_{\Lambda
_N}=(k\rho_{k,\Lambda_N}/\rho)_{k\in\N}$ under $\P_{\beta_N,\Lambda
_N}^\ssup{N}$ with $\Lambda_N=[0,L_N]^d$ satisfies an LDP with speed
$N\beta_N$ and rate function given by $g_\nu(\cdot)-\mu(\nu)$. This
would be in line with Theorems~\ref{thm:ldp} and~\ref{thm:main},
and we do believe that this is indeed true, but we make no attempt to
prove this.


\subsection{Limit laws in the dilute low-temperature limit}\label{sec-Limitlaws}
The minimiser(s) of the rate function $f(\beta,\rho,\cdot)$ are of high
interest, since they describe the limiting behaviour of the cluster
size distribution under the Gibbs measure. It is a general fact from
the theory of $\Gamma$-limits that $\Gamma$-convergence implies the
convergence of minima over compact subsets and of the minimiser(s). For
the limiting rate function $g_\nu$, the global minimiser has been
identified in \cite{CKMS10}. The minimum is
%
%
\begin{equation}
\label{muident} \mu(\nu) = \inf_{\mathcal{Q}} g_\nu= \inf
_{N\in\N} \frac{E_N - \nu}{N},
\end{equation}
and the minimisers are given as follows.

\begin{lemma}[(Minimisers of $g_\nu$)] \label{12345678987456321}
The number $\nu^*:= \inf_{N\in\N} (E_N - N e_\infty)$ is strictly
positive. The map $\nu\mapsto\mu(\nu)$ is continuous, piecewise affine
and concave. Let $\mathcal{N} \subset(0,\infty)$ be the set of points
where $\mu(\cdot)$ changes its slope. Then $\Ncal$ is bounded, and $\mu
(\nu)=-\nu$ for $\nu>\max\Ncal$ and $\mu(\nu)=e_\infty$ for $\nu<\nu
^*$. Furthermore:

\begin{longlist}[(1)]

\item[(1)]$\nu^*\in\mathcal{N} \subset[\nu^*,\infty)$, and $\mathcal{N}$
is at most countable with $\nu^*$ as only
possible accumulation point.

\item[(2)] For $\nu> \nu^*$, we have $\mu(\nu)< e_\infty$ and every
minimiser $\vect{q} = (q_k)_{k}$ of $g_\nu$ satisfies $\sum_{k\in\N}
q_k =1$. If $\nu\notin\mathcal{N}$, then $g_\nu$ has the unique
minimiser $\vect{q}^\ssup{\nu}=(q_k^\ssup{\nu})_k$ with $q_k^\ssup{\nu}
= \delta_{k,k(\nu)}$ with $k(\nu)$ the unique minimiser of $k\mapsto
(E_k -\nu)/k$ over $\N$. The map $\nu\mapsto k(\nu)$ is constant
between subsequent points in $\Ncal$.

\item[(3)] For $\nu<\nu^*$, we have $\mu(\nu) = e_\infty$ and the unique
minimiser of $g_\nu$ is the constant zero sequence $(q_k)_{k\in\N}$
with $q_k = 0$ for any $k$.
\end{longlist}
\end{lemma}

This is essentially \cite{CKMS10}, Theorem~1.5; the proof is found in
the \hyperref[sec-Appendix]{Appendix}. If, as in \cite{CKMS10}, the
point $\infty$ is added to
the state space $\N$ of the measures in $\mathcal{Q}$, then the
minimisers of $g_\nu$ are concentrated on $\N$ for $\nu>\nu^*$ and on
$\{\infty\}$ for $\nu<\nu^*$; it was left open in \cite{CKMS10} whether
or not the latter regime is nonvoid.

The set $\mathcal{N}$ is infinite if and only if $(E_k - k e_\infty
)_{k\in\N}$ has no minimiser. In dimensions $d\geq2$, it is expected
(and shown in some cases; see~\cite{R81,yeung}) that $E_k - k
e_\infty\geq\const k ^{1-1/d} \to\infty$, ensuring that $\mathcal
{N}$ is a finite set.

Now we can draw a conclusion from Theorem~\ref{thm:main} about the
limiting behaviour of the minimisers of the rate function in the dilute
low-temperature limit. The following assertions are well known
consequences from the $\Gamma$-convergence of Theorem~\ref{thm:main};
see \cite{dMaso}, Theorem~7.4 and Corollary 7.24.

%
\begin{cor}\label{Cor-MinimConv} In the situation of Theorem~\ref{thm:main}:
\begin{longlist}[(1)]
\item[(1)] the free energy per particle converges to $\mu(\nu)$,
\[
\frac{1}\rho f(\beta,\rho) \to\mu(\nu);
\]
%
%
\item[(2)] if $\mu(\cdot)$ is differentiable at $\nu$ [i.e., for $\nu\in
(0,\infty)\setminus\Ncal$], any minimiser $\vectt{\rho}^\ssup{\beta,\rho
}= (\rho_k^\ssup{\beta,\rho})_k$ of $J_{\beta,\rho}$ converges to the
minimiser of $g_\nu$:
\[
\frac{k \rho_k^\ssup{\beta,\rho}} {\rho} \to q_k^\ssup{\nu},\qquad k\in \N.
\]
\end{longlist}
\end{cor}

%
\begin{figure}[b]

\includegraphics{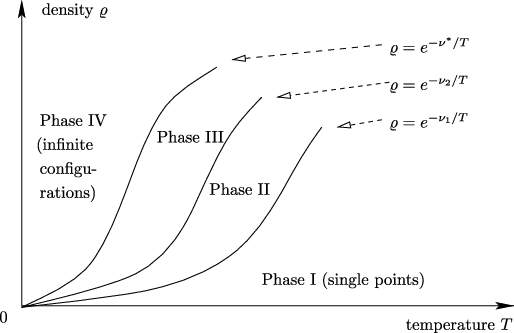}

\caption{A diagram illustrating the expected relationship of the slope
condition $-T \log\rho=   -\beta^{-1} \log\rho\to\nu$
and the minimisers of the rate function in the dilute low-temperature
limit. The phases I and IV always exist.
Depending on the pair potential, there can be values of $\nu$ for which
$2 \leq k(\nu) <\infty$, yielding intermediate phases
(e.g., II and III, although the precise number of intermediate phases
will depend on the pair potential).}\label{fig4}
\end{figure}

For an illustration, see Figure \ref{fig4}.
Another important consequence of Theorem~\ref{thm:main}, together with
the LDP of Theorem~\ref{thm:ldp}, is a kind of law of large numbers for
the cluster size distribution $\vectt{\rho}_{\Lambda_N}$ in the
thermodynamic limit, followed by the low-temperature dilute limit. A
convenient formulation is in terms of the vector $\vect{q}_\Lambda
=(q_{k,\Lambda})_{k\in\N}$ with $q_{k,\Lambda}=k\rho_{k,\Lambda}/\rho$,
the frequency of particles in $k$-clusters, if $|\Lambda|=N/\rho$.

%
\begin{cor}\label{Cor-LLN}
For any $\nu\in(0,\infty)\setminus\Ncal$, any $K\in\N$ and any $\eps
>0$, if $\beta$ is sufficiently large, $\rho$ sufficiently small and
$-\frac{1}\beta\log\rho$ is sufficiently close to $\nu$, then, for boxes
$\Lambda_N$ with volume $N/\rho$,
%
%
\begin{equation}
\label{LLN>} \lim_{N\to\infty}\P_{\beta,\Lambda_N}^\ssup{N}
\bigl(|q_{k(\nu),\Lambda
_N}-1|\geq\eps \bigr)=0\qquad \mbox{if }\nu>\nu^*
\end{equation}
and
%
%
\begin{equation}
\label{LLN<} \lim_{N\to\infty}\P_{\beta,\Lambda_N}^\ssup{N}
\Biggl(\sum_{k=1}^K q_{k,\Lambda_N}\geq
\eps \Biggr)=0 \qquad\mbox{if }\nu<\nu^*.
\end{equation}
\end{cor}

\begin{pf} We prove \eqref{LLN>} and \eqref{LLN<} simultaneously.
Consider the set
\[
A= %
\cases{ \displaystyle\biggl\{\vectt{\rho}\in M_\rho\dvtx \biggl|
\frac{k(\nu)\rho_{k(\nu)}}\rho -1 \biggr|\geq\eps \biggr\},&
\quad$\mbox{for } \nu>\nu^*,$\vspace*{2pt}
\cr
\displaystyle\Biggl\{\vectt{\rho}\in M_\rho\dvtx\sum
_{k=1}^K \frac{k\rho_{k,\Lambda
}}\rho\geq\eps \Biggr\},& \quad$\mbox{for } \nu<\nu^*$.} %
\]
Then the $\Gamma$-convergence of Theorem~\ref{thm:main} implies \cite{dMaso},
Theorem~7.4, that
\[
\liminf_{\beta,\rho}\frac{1}\rho\inf_A
f(\beta,\rho,\cdot)\geq-\inf_A g_\nu,
\]
where $\liminf_{\beta,\rho}$ refers to the limit in Theorem~\ref
{thm:main}. Furthermore, it is easy to see from Lemma~\ref{12345678987456321}
that $\delta=\inf_A g_\nu-\inf g_\nu$ is positive. We pick now $\beta$
so large and $\rho$ so small and $-\beta^{-1}\log\rho$ so close to $\nu
$ that $\frac{1}\rho\inf_A f(\beta,\rho,\cdot)-\inf_A g_\nu\geq-\delta
/4$ and $\frac{1}\rho f(\beta,\rho)-\mu(\nu)\leq\delta/4$ [the latter is
possible by Corollary~\ref{Cor-MinimConv}(1)]. Now the LDP of
Theorem~\ref{thm:ldp} yields that
\begin{eqnarray*}&& \limsup_{N\to\infty}\frac{1}{|\Lambda_N|}
\log\P_{\beta,\Lambda_N}^\ssup {N} (\vectt{\rho}_{\Lambda_N}\in A ) \\
&&\qquad
\leq-\inf_A I_{\beta,\rho}=-\beta \Bigl[\inf
_A f(\beta,\rho,\cdot )-f(\beta,\rho) \Bigr]
\\
&&\qquad\leq-\beta\rho \biggl[\inf_A g_\nu-\mu(\nu)-
\frac\delta4-\frac\delta 4 \biggr]=-\beta\rho\delta/2<0.
\end{eqnarray*}
Hence, $\lim_{N\to\infty}\P_{\beta,\Lambda_N}^\ssup{N}(\vectt{\rho
}_{\Lambda_N}\in A)=0$. Noting that this probability is identical to
the two probabilities on the left of \eqref{LLN>} and \eqref{LLN<} for
our two choices of $A$, finishes the proof.
\end{pf}

It may come as a surprise that, for most values of the parameter $\nu$,
the cluster size distribution is asymptotically concentrated on just
one particular cluster size that depends only on $\nu$. This may be
vaguely explained by the fact that the zero-temperature limit $\beta\to
\infty$ forces the system to become asymptotically ``frozen'' in a
state in which every cluster size assumes the globally optimal
configuration size, which is unique for $\nu\in(\nu^*,\infty)\setminus
\Ncal$. Furthermore, note that Corollary~\ref{Cor-LLN} does not give
the existence of ``infinite large'' clusters (i.e., clusters whose size
diverges with $N$) for any value of $\beta$ and $\rho$, not even for
$\nu<\nu^*$ and $-\beta^{-1}\log\rho\approx\nu$.

\subsection{Uniform bounds}\label{sec-Uniform}
Under some natural additional assumptions on the pair potential, the
assertions of
Theorem~\ref{thm:main} can be strengthened; see Theorem~\ref{thm:unif}
below. Indeed,
we will assume that the ground states of the functional $U_N$ consist
of well-separated particles,
which are contained in a ball with volume of order $N$, and we assume
some more regularity of the
interaction function $v$. Then we show that the $\Gamma$-convergence
in Theorem~\ref{thm:main} in the
coupled limit in \eqref{dilutelimit} can be strengthened to estimates
that are uniform in some low-temperature,
low-density rectangle $(\overline\beta,\infty) \times(0,\overline\rho)$.
This leads to corresponding uniform estimates
on $|\frac{1}\rho f(\beta,\rho) - \mu(\nu)|$ and on minimisers. We now
formulate this.

%
\begin{assumption}[(Minimum interparticle distance, H{\"o}lder
continuity)] \label{ass:hoelder}

\begin{longlist}[(ii)]
\item[(i)] There is $r_\mathrm{min}\geq r_\mathrm{hc}$ such that, for
all $N \in\N$,
every minimiser $(x_1,\ldots,\break x_N) \in(\R^d)^N$ of the energy function
$U_N$ has interparticle
distance lower bounded as $|x_i- x_j| \geq r_\mathrm{min}$, $i\neq j$.
\item[(ii)] The pair potential $v$ is uniformly H{\"o}lder continuous
in $[r_\mathrm{min},\infty)$.
\end{longlist}
\end{assumption}

The existence of a uniform lower bound $r_\mathrm{min}$ for ground
state interparticle distance is, of course,
trivial when the potential has a hard core $r_\mathrm{hc}>0$. A
sufficient condition for
the existence of $r_\mathrm{min}>0$ for a potential without
hard core is, for example, that $v(r)/r^d \to\infty$ as $r\to0$, as
can be shown along~\cite{Th06}, Lemma 2.2.

%
\begin{assumption}[(Maximum interparticle distance)] \label{ass:maxdist}
There is a constant $c>0$ such that for all $N\in\N$
every minimiser $(x_1,\ldots,x_N) \in(\R^d)^N$ of the energy function $U_N$
has interparticle distance upper bounded by $|x_i - x_j| \leq c N
^{1/d}$.
\end{assumption}

On physical grounds, we would expect that this assumption is
true for every reasonable potential. To the best of our knowledge,
however, nontrivial rigorous results are available
in dimension two only, for Radin's soft disk potential~\cite{R81} and
for potentials satisfying
conditions (H1) to (H3) from~\cite{yeung}. These potentials satisfy
Assumption~\ref{ass:hoelder} as well.

%
\begin{theorem} \label{thm:unif}
Suppose that in addition to Assumption \ref{assV} the pair potential also
satisfies Assumptions~\ref{ass:hoelder} and~\ref{ass:maxdist}. Then
there are $\overline\rho,\overline\beta,C>0$ such that for every
$(\beta,\rho) \in[\overline\beta,\infty)\times(0,\overline\rho]$,
putting $\nu: = - \beta^{-1} \log\rho$, the following holds:
\begin{longlist}[(1)]
\item[(1)] Estimate on the rate function:
%
%
\begin{equation}
\label{estiratefct}\qquad \biggl| \frac{1} \rho f\biggl(\beta,\rho, \biggl({
\frac{\rho q_k}{k}\biggr)_{k\in\N}}\biggr) - g_\nu
\bigl((q_k)_k \bigr)\biggr | \leq\frac{C}\beta\log
\beta,\qquad (q_k)_{k\in\N}\in\mathcal{Q}.
\end{equation}
%
%
\item[(2)] Estimate on the free energy:
%
%
\begin{equation}
\label{estifreeen} \biggl|\frac{1}\rho f(\beta,\rho) - \mu( \nu) \biggr| \leq2
\frac{C} \beta \log\beta.
\end{equation}
\item[(3)] Minimisers: For any minimiser $\vectt{\rho}^\ssup{\beta,\rho}$ of
$f(\beta,\rho,\cdot)$, put $\vect{q}^\ssup{\beta,\rho}:=\break
( k \rho_k^\ssup{\beta,\rho}/  \rho)_{k\in\N}$. Then, if $\nu<\nu^*$,
%
%
\begin{equation}
\label{eq:smallnu} \sum_{k\in\N} \frac{q_k^\ssup{\beta,\rho}}{k} \leq2
\frac{C} {\nu^* -
\nu} \frac{1}\beta\log\beta.
\end{equation}
If $\nu>\nu^*$, then
%
%
\begin{equation}
\label{eq:largenu} \sum_{k\in M(\nu)} q^\ssup{\beta,
\rho}_k \geq1 - 2 \frac{C} {\Delta(\nu
)} \frac{1}\beta\log\beta,
\end{equation}
where
\[
\Delta(\nu):= \inf \biggl\{ \frac{E_k- \nu}{k} \Big| k \in\N \setminus M(\nu) \biggr
\}- \mu(\nu)>0
\]
is the gap above the minimum, and $M(\nu) \subset\N$ is the
set of minimisers of $((E_k - \nu)/k)_{k\in\N}$ [thus $M(\nu) = \{
k(\nu)\}$ for $\nu\notin\mathcal{N}$].
\end{longlist}
\end{theorem}

Theorem~\ref{thm:unif} is proved in Section~\ref{sec-proofThm17}. One
can see from the proof that one can choose $\overline\rho=(2\alpha
+2R)^{-d}$. It follows in particular that the $\Gamma$-convergence and
the two convergences from Corollary~\ref{Cor-MinimConv} can be
strengthened to convergence for just taking $\beta\to\infty$, uniformly
in $\rho\in(0,\overline\rho]$, with an error of order $\beta^{-1} \log
\beta$. This form of the error order term is an artefact of the
assumption of H\"older continuity; the constant $C$ depends on the H\"
older parameter.

Note that \eqref{eq:smallnu} implies that, in the case $\nu<\nu^*$,
for every $K \in\N$, the
fraction of particles in clusters of size $ \leq K$ is bounded by
\[
\sum_{k\leq K} \frac{k\rho_k^\ssup{\beta,\rho}}{\rho} = \sum
_{k\leq K} q_k^\ssup{\beta,\rho} \leq
\frac{2C}{\nu^*- \nu} K \frac{1}\beta\log\beta.
\]
This shows that, as $\beta\to\infty$, for some choices of $K=K_\beta\to
\infty$, the fraction of particles in clusters of size $\leq K_\beta$
vanishes; that is, the average cluster size becomes very large. Note
that the law of large numbers in \eqref{LLN<} in Corollary~\ref
{Cor-LLN} may, under Assumptions~\ref{ass:hoelder} and \ref
{ass:maxdist}, be proved also with $K$ replaced by $K_\beta$.

\subsection{Some remarks concerning related mathematical and physical
problems}\label{sec-physics} Our problem is connected with continuum
percolation problems for interacting particle systems; see the
review~\cite{ghm01}. In our setting of finite systems, the term
``percolation'' should be replaced with ``formation of unbounded
components,'' that is, clusters whose size diverges as the number of
particles goes go infinity. The problem of percolation or
nonpercolation for continuous particle systems in an infinite-volume
Gibbs state (i.e., in a grand-canonical setting) is studied in~\cite
{pechersky}, where Pechersky and Yambartsev prove that, for
sufficiently high chemical potential and sufficiently low temperature,
percolation does occur. However, they do not give any information on
the densities at which percolation occurs. This hinders the physical
interpretation, since one cannot say whether the percolation is due to
high density or strong attraction. In this light, our results are
stronger and at the same time weaker: we
do show that a transition from bounded to unbounded clusters happens at
low density, but only in a limiting sense along low-temperature,
low-density curves; there is no fixed temperature or density at which
we prove the formation of unbounded clusters.

In addition, our work has an interesting relationship to quantum
Coulomb systems. In the simplest case, a gas of protons and electrons,
we may ask whether we observe a fully ionized gas, where protons
and electrons stay for themselves, or a gas of neutral molecules, with
protons and electrons paired up together. The mathematical model is in
terms of a quantum mechanical Hamiltonian in a fermionic Hilbert space
for particles of positive and negative charge interacting via a long
ranged Coulomb potential. The analogues of our ground states $E_k$ are
defined as ground state energies of the Hamiltonian restricted to
sectors with prescribed particle numbers and center of mass motion removed.
Rigorous mathematical results were given by Fefferman~\cite{fefferman}
(see also~\cite{cly}), in the \emph{Saha regime},
also called \emph{atomic or molecular limit}: when the temperature goes
to $0$ at fixed, negative enough
chemical potential, the Coulomb gas behaves like an ideal gas of
different types of molecules or particles. The
chemical composition is determined by the chemical potential by an
energy minimization problem, akin to minimizing $E_k - k \mu$ as a
function of $k$, which in turn is the grand-canonical version of our
auxiliary variational problem $ (E_k - \nu)/k = \mathrm{min}$.

Our results adapt this quantum Coulomb system picture to a classical
setting. From this point of view,
the key novelty is that we work in the canonical rather than the
grand-canonical ensemble; this allows us to extend results to the
region where formation of large clusters occurs. Indeed, in the
canonical ensemble we can take the density larger than the transition
density, which is conjectured to exist and to be of the order $\exp
(-\beta\nu^*)$, and at the same time impose that the density be small.
In the grand-canonical ensemble, we may of course take the chemical
potential larger than the transition potential, but then we lose
control over the density and cannot apply the dilute mixture approximation.

The remainder of this paper is organised as follows. In Section~\ref
{sec-ProofLDP} we prove the LDP of Theorem~\ref{thm:ldp}, in
Section~\ref{sec-boundsf} we compare the rate function with an explicit
ideal rate function, and in Section~\ref{sec-finftyfinite} we compare
temperature-depending quantities with the ground states. Finally, the
proofs of Theorems~\ref{thm:main} and \ref{thm:unif} are given in
Section~\ref{sec-lastproofs}.

\section{Proof of the LDP}\label{sec-ProofLDP}

In this section, we prove Theorem~\ref{thm:ldp}. We fix $\beta
\in(0,\infty)$ and $\rho\in(0,\rho_{\mathrm{ cp}})$ throughout this section.
In Section~\ref{sec-Strategy} we explain our strategy and formulate the
main steps, and in Sections~\ref{sec-proofsubaddi}--\ref{sec-proofLDPj}
we prove these steps. The proof of Theorem~\ref{thm:ldp} is finished in
Section~\ref{sec-finishLDP}.

\subsection{Strategy}\label{sec-Strategy}

The main idea is to derive first a large deviations principle
for the distribution of $(\rho_{k,\Lambda})_{k=1,\ldots,j}$ for fixed
$j\in\N$, that is, for the projection of $\vectt{\rho}_{\Lambda}$ on the
first $j$ components, and apply the Dawson--G\"artner theorem for the
transition to the projective limit as $j\to\infty$. From the proof of
the principle for the projection, we isolate an important step (see
Proposition~\ref{prop:constrained-free-energy}): using standard
subadditivity arguments, we prove the existence of thermodynamic limit
for constrained free energy, the constraint referring to cluster size
concentrations of size $\leq j$. The principle for the projection of
$\vectt{\rho}_{\Lambda}$ appears in Proposition~\ref{prop:ldp-j}.

%

Given $N,N_1,\ldots,N_j \in\N_0$ define the constrained partition
function with fixed cluster numbers of size $\leq j$,
%
%
\begin{equation}
\label{eq:constrained-pf} \qquad Z_\Lambda(\beta,N,N_1,\ldots,N_j):= \frac{1}{N!} \int_{\Lambda^N} \e^{-\beta U_N(\vect{x})} \prod
_{k=1}^j\1\bigl\{N_k(\vect{x})
= N_k\bigr\} \,\dd\vect{x}.
\end{equation}
Note that $Z_\Lambda(\beta,N,N_1,\ldots,N_j)=0$ if $\sum_{k=1}^j k N_k >N$.

In the following we shall often be interested in the
interior or boundary of subsets $A \subset[0,\infty)^{j+1}$ for some
$j \in\N$.
Unless explicitly stated otherwise, $\Int A$ and $\partial A$ refer to
the interior
and boundary of $A$ considered as a subset of $\R^{j+1}$. In
particular, if $0 \in A$, then
$0$ is automatically a boundary point.

We denote by $\dom h= \{x\dvtx h(x)<\infty\}=\{ h(\cdot) <\infty\}$
the effective domain of an $(-\infty,\infty]$-valued function $h$.

%
\begin{prop} \label{prop:constrained-free-energy}
Fix $j\in\N$. Then there is a function $f_j(\beta,\cdot)\dvtx\break
[0,\infty)^{j+1} \to\R\cup\{\infty\}$ such that:
\begin{itemize}
\item$f_j(\beta,\cdot)$ is convex and lower semi-continuous;
\item its effective domain has nonempty interior $\Delta_j:=\Int_{\R
^{j+1}} \dom f_j(\beta,\cdot)$
and $f_j(\beta,\cdot)$ is continuous in $\Delta_j$;
\item its effective domain is contained in
\[
\dom f_j(\beta,\cdot) \subset\overline\Delta_j \subset
\Biggl\{ (\rho,\rho_1,\ldots, \rho_j) \in[0,
\infty)^{j+1} \Big| \rho\in[0,\rho_\mathrm{cp}], \sum
_{k=1}^j k \rho_k \leq\rho \Biggr\},
\]
\end{itemize}
and, moreover, if $|\Lambda_N|, N, N_1^\ssup{N},\ldots, N_j^\ssup{N}
\to\infty$ in such a way that
%
%
\begin{equation}
\label{convN1bisj} \frac{N}{|\Lambda_N|}\to\rho, \frac{N_1^\ssup{N}}{|\Lambda_N|} \to
\rho_1, \ldots, \frac{N^\ssup{N}_j}{|\Lambda_N|} \to\rho_j,
\end{equation}
then:
\begin{itemize}
\item if $(\rho,\rho_1,\ldots,\rho_j) \in\Delta_j$,
%
%
\begin{equation}
\label{eq:cfe-limit} \lim_{N\to\infty} \frac{1}{|\Lambda_N|}\log
Z_{\Lambda_N}\bigl(\beta,N,N_1^\ssup{N},
\ldots,N_j^\ssup{N}\bigr) = - \beta f_j(\beta,
\rho,\rho_1,\ldots,\rho_j),
\end{equation}
and the limit is finite;
\item if $(\rho,\rho_1,\ldots,\rho_j) \in\partial\Delta_j$ (boundary of
$\Delta_j$), then
%
%
\begin{eqnarray}
\label{eq:cfe-lim-boundary} \limsup_{N\to\infty} \frac{1}{|\Lambda_N|} \log
Z_{\Lambda_N}\bigl(\beta,N,N_1^\ssup{N},
\ldots,N_j^\ssup{N}\bigr) &\leq&- \beta f_j(\beta,
\rho,\rho_1,\ldots,\rho_j)
\nonumber
\\[-8pt]
\\[-8pt]
\nonumber
&\in&\R\cup\{-\infty \};
\end{eqnarray}
%
%
\item if $(\rho,\rho_1,\ldots,\rho_j)\in\overline{\Delta}_j^{\mathrm c}$,
then \eqref{eq:cfe-limit} holds true, and the limit is $- \beta
f_j(\beta,\rho,\rho_1,\break\ldots,  \rho_j) = -\infty$.
\end{itemize}
\end{prop}

This proposition is proved in Section~\ref{sec-proofsubaddi}.

The set $\Delta_1$ is related to close-packing situations. For example,
when $j=1$ and the density
$\rho$ is higher than $1/|B(0,R)|$ (where we recall that $R$ is the
parameter in our notion of connectedness),
it is impossible to have a gas formed only of $1$-clusters, and we have
$f_1(\beta,\rho,\rho) = \infty$.

%

Analogously to \eqref{eq:rf-fe}, let
\[
\label{eq:ij} I_{\beta,\rho,j}(\rho_1,\ldots,\rho_j):= \beta \bigl( f_j(\beta,\rho,\rho_1,\ldots,
\rho_j) - f(\beta,\rho) \bigr).
\]
We will prove in Section~\ref{sec-proofLDPj} the following.

%
\begin{prop}[(LDP for projection of $\vectt{\rho}_{\Lambda}$)] \label{prop:ldp-j}
Fix $j\in\N$. Then, in the thermodynamic limit $N\to\infty$, $L\to
\infty$, $N/L^d \to\rho$, the distribution of $(\rho_{1,\Lambda},\ldots,   \rho_{j,\Lambda})$ under the Gibbs measure $\P_{\beta,\Lambda}^\ssup
{N}$ with $\Lambda=[0,L]^d$
satisfies a large deviations principle with scale $|\Lambda|$
and rate function $I_{\beta,\rho,j}$. Moreover,
the rate function is good and convex.
\end{prop}

Recall that a rate function is called good if its level sets are
compact. In this case, it is in particular lower semicontinuous. The
large deviations principle means that, for any open set $\Ocal\subset
[0,\infty)^j$ and any closed set $\Ccal\subset[0,\infty)^j$, with
$\Lambda=[0,L]^d$,
%
%
\begin{eqnarray}
\liminf_{N,L\to\infty, N/L^d\to\rho}\frac{1}{|\L|}\log\P_{\beta,\Lambda
}^\ssup{N}
\bigl((\rho_{1,\Lambda},\ldots, \rho_{j,\Lambda})\in\Ocal \bigr)&\geq& -
\inf_{\Ocal}I_{\beta,\rho,j},\label{lowerbound}
\\
\limsup_{N,L\to\infty, N/L^d\to\rho}\frac{1}{|\L|}\log\P_{\beta,\Lambda
}^\ssup{N}
\bigl((\rho_{1,\Lambda},\ldots, \rho_{j,\Lambda})\in\Ccal \bigr)&\leq& -
\inf_{\Ccal}I_{\beta,\rho,j}.\label{upperbound}
\end{eqnarray}
We refer to \eqref{lowerbound} as to the \textit{lower bound for open
sets} and to \eqref{upperbound} as to the \textit{upper bound for
closed sets}.

\subsection{Proof of Proposition~\texorpdfstring{\protect\ref{prop:constrained-free-energy}}{2.1}: Subadditivity arguments}\label{sec-proofsubaddi}

In this section we prove Proposition~\ref
{prop:constrained-free-energy}. For the remainder of this section, we
fix $j\in\N$.

The crucial point is the following supermultiplicativity of partition functions,
which translates into subadditivity of free energies:
Let $N',N''\in\N$. Let $\Lambda', \Lambda''$ be two disjoint
measurable sets
which have mutual distance larger than the potential range $b$,
and $\Lambda$ large enough to contain the union of the two. Then
%
%
\begin{equation}
\label{eq:supermult} Z_{\Lambda}\bigl(\beta,N'+N''
\bigr) \geq Z_{\Lambda'\dot\cup\Lambda''}\bigl(\beta,N'+N''
\bigr) \geq Z_{\Lambda'}\bigl(\beta,N'\bigr) Z_{\Lambda''}
\bigl(\beta,N''\bigr).
\end{equation}
This standard trick leads to a proof of the existence of the
thermodynamic limit
by subadditivity methods \cite{Ru99} (where subadditivity is applied
to the microcanonical ensemble instead of canonical, but the method is
the same).

The starting point of our proof is the observation that a similar inequality
holds for constrained partition functions $Z_{\Lambda}(\beta,N,N_1,\ldots,N_j)$
provided $\Lambda'$ and $\Lambda''$ have mutual distance $>R$,
where we recall that $R\in(b,\infty)$ was picked arbitrarily.
Therefore we can prove existence of the constrained free energy by adapting
the standard methods. Let us recall, roughly, the standard strategy of proof:
\begin{longlist}[(1)]
\item[(1)] As a first step, one proves existence of limits of $-\frac{1}{\beta
|\Lambda|}\log Z_{\Lambda}(\beta,N,N_1,\break\ldots,  N_j)$ along special
sequences of cubes---roughly, the sequence is defined in an iterative way
by doubling the cube's side length and adding a ``security margin,''
and multiplying particle numbers by $2^d$. This uses subadditivity
and yields a densely defined, convex function $\eta$.
\item[(2)] Then one shows that the function $\eta$ is locally bounded in some
region of nonempty interior
and therefore can be extended to a continuous function $f$ in some nonempty
open set $\Delta$.
\item[(3)] At last, one proves the convergence of $-\frac{1}{\beta|\Lambda
|}\log Z_{\Lambda}(\beta,N,N_1,\ldots,N_j)$ to $f$ along general cubes.
\end{longlist}

Our proof follows these steps, with some complications. Notably, an extra
argument is required in step (2); see Lemma~\ref{lem:nonempty-int} below.
Moreover, we make the choice---convenient
in view of the large deviations framework---to assign values to the free
energy not only in $\Delta$ and outside $\overline\Delta$ (where $f$
is $\infty$) but also in $\partial\Delta$ by requiring global
lower semi-continuity and convexity.

\subsubsection{Convergence along special sequences}\label{sec-specialcubes}

Let $R'>R$ and $L^*_0>0$ be fixed, and define $(L^*_n)_{n\in
\N_0}$ recursively by
$L^*_{n+1} = 2 L^*_n + R'$. Explicitly, $L^*_n =- R' + 2^n (L^*_0 + R')$.
Let $\Lambda^*_n = [0,L^*_n]^d$.
Thus $\Lambda^*_{n+1}$ can be considered as the union
of $2^d$ copies of $\Lambda_n$ with a corridor of width $R'$ between them.
Let
\begin{eqnarray*}
&&\mathcal{D}_j:= \bigl\lbrace \vectt{\rho}=(\rho,\rho_1,
\ldots,\rho_j) \in[0,\infty)^{j+1} \mid\\
&&\hspace*{34pt}\rho >0, \exists q \in
\N_0\dvtx 2^{qd} \bigl(L^*_0+R'
\bigr)^d \vectt{\rho} \in\N_0^{j+1} \bigr
\rbrace. 
\end{eqnarray*}

%
\begin{lemma}[{[Introduction of $\eta_j(\beta,\cdot)$]}]\label{lem-etadef}
Let $(\rho,\rho_1,\ldots,\rho_j) \in\mathcal{D}_j$, and put for $n\in
\N$
%
%
\begin{eqnarray}
\label{eq:special-ns} N^\ssup{n}:= 2^{nd} \bigl(L^*_0 +
R'\bigr)^d \rho,\qquad N_k^\ssup{n}:=
2^{nd} \bigl(L^*_0 + R'\bigr)^d
\rho_k
\nonumber
\\[-8pt]
\\[-8pt]
 \eqntext{(k=1,\ldots,j).}
\end{eqnarray}
The following limit exists in $\R\cup\{\infty\}$ and is equal to an infimum:
%
%
\begin{eqnarray}
\label{etadef} %
\eta_j(\beta,\rho,
\rho_1,\ldots,\rho_j)&:=& - \lim_{n\to\infty}
\frac{1}{\beta|\Lambda^*_n|} \log Z_{\Lambda^*_n}\bigl(\beta, N^\ssup{n},N_1^\ssup{n},
\ldots, N_j^\ssup{n}\bigr)
\nonumber
\\[-8pt]
\\[-8pt]
\nonumber
& =& \inf_{n\in\N} \biggl( - \frac{1}{\beta|\Lambda_n^*|} \log
Z_{\Lambda^*_n}\bigl(\beta, N^\ssup{n},N_1^\ssup{n},
\ldots, N_j^\ssup{n}\bigr) \biggr).
\end{eqnarray}
This limit is finite as soon as $Z_{\Lambda^*_n}(\beta, N^\ssup
{n},N_1^\ssup{n},\ldots, N_j^\ssup{n})>0$ for some $n\in\N$. In particular,
%
%
\begin{equation}
\label{eq:etadom} \bigl\{\eta_j(\beta,\cdot)<\infty\bigr\}\subset \Biggl
\{(\rho,\rho_1,\ldots,\rho _j) \in\Dcal_j\dvtx
\sum_{k=1}^j k \rho_k\leq\rho
\leq\rho_{\mathrm{cp}} \Biggr\}.
\end{equation}
\end{lemma}

\begin{pf} We can place $2^d$ shifted copies of $\Lambda^*_n$ in
$\Lambda^*_{n+1}$ in such a way that the copies have distance $\geq R'$
to each other. Hence we have
\[
Z_{\Lambda^*_{n+1}} \bigl(\beta,N^{\ssup{n+1}}, N_1^{\ssup{n+1}},
\ldots, N_j^{\ssup{n+1}} \bigr) \geq \bigl( Z_{\Lambda^*_{n}} \bigl(
\beta,N^\ssup n, N_1^\ssup n,\ldots,
N_j^{\ssup{n}} \bigr) \bigr)^{2^d}.
\]
Abbreviating
\[
u_n=-\frac{1}{|\Lambda^*_{n}|} \log Z_{\Lambda^*_{n}} \bigl(
\beta,N^\ssup n, N_1^\ssup n,\ldots,
N_j^{\ssup{n}} \bigr)\quad \mbox{and} \quad 1+\eps_n:=
\frac{2^d |\Lambda^*_n|}{|\Lambda^*_{n+1}|},
\]
this is just the inequality $u_{n+1} \leq(1+\eps_n) u_n$. Our goal is
to show that $\lim_{n\to\infty}u_n$ exists and is equal to $\underline
u:= \inf_{n\in\N}u_n$. Note that
\[
1+\eps_n=\frac{2^d |\Lambda^*_n|}{|\Lambda^*_{n+1}|} = \biggl( \frac{2^{n+1}(L_0^* + R') - 2 R'}{
2^{n+1} (L_0^* + R') - R' }
\biggr)^d = 1+ O\bigl(2^{-n}\bigr),
\]
which yields $\sum_{n=1}^\infty|\eps_n| <\infty$. The case $\underline
u = - \infty$ is excluded by exploiting the stability of the energy:
for some $C\in(0,\infty)$, we have
\begin{eqnarray*}
Z_{\Lambda_n^*} \bigl(\beta,N^\ssup{n}, N_1^\ssup
n,\ldots, N_j^{\ssup
{n}} \bigr) &\leq& Z_{\Lambda_n^*}\bigl(
\beta,N^{\ssup n}\bigr) \leq\frac{1}{N^\ssup{n}!} \e^{- \beta E_{N^\ssup{n}}}
\bigl|\Lambda_n^*\bigr|^{N^\ssup{n}}
\\
&\leq& \e^{C N^\ssup{n}},
\end{eqnarray*}
and hence $\underline u\geq-C\rho$.

If $\underline u = \infty$, then $u_n = \infty$ for all $n$ and in
particular $u_n \to\infty=\underline u$. Consider now the case
$\underline u\in\R$. For $\delta>0$, let $q \in\N$ such that $u_q
\leq\ell+ \delta$ and $ 1 - \delta\leq\prod_{k=q}^n(1+\eps_k) \leq
1+\delta$ for all $n \geq q$. Then for $n \geq q$,
\[
\underline u \leq u_n \leq u_q \prod
_{k=q}^{n-1} (1+\eps_k) \leq (\underline u
+\delta) (1+\delta).
\]
Letting first $n\to\infty$ and then $\delta\to0$ we conclude that
$u_n \to\underline u$. The additional assertion is clear from the
proof and from the fact that, for $\rho>\rho_{\mathrm{ cp}}$, we have $\infty
=f(\beta,\rho)=-\frac{1}\beta\lim_{n\to\infty}\frac{1}{|\Lambda_n^*|}\log
Z_{\Lambda_n^*}(\beta,N^{\ssup n})$.
\end{pf}

\subsubsection{Properties of the limit function \texorpdfstring{$\eta_j(\beta,\cdot)$}{$eta_j(beta,cdot)$}}

The next lemma essentially states that $\eta_j(\beta,\cdot)$
is a convex function. The precise formulation needs some care since the
domain $\Dcal_j$ of this function is not closed under taking arbitrary
convex combinations.

%
\begin{lemma}
\label{lem:etaconv}
Let $\vectt{\rho},\vectt{\rho}' \in\mathcal{D}_j$. Let $t\in(0,1)$ be
a dyadic fraction, that is, of the form $t = p / 2^q$ for
some $p,q\in\N_0$. Then $t \vectt{\rho} + (1-t) \vectt{\rho}' \in
\mathcal{D}_j$ and
%
%
\begin{equation}
\label{eq:etaconv} \eta_j\bigl(\beta,t \vectt{\rho} + (1-t) \vectt{
\rho}'\bigr) \leq t \eta_j(\beta,\vectt{\rho}) + (1-t)
\eta_j\bigl(\beta,\vectt{\rho}'\bigr).
\end{equation}
\end{lemma}

\begin{pf}
Consider the cubes $\Lambda_n^*$ defined as above. $\Lambda_{n+1}^*$
is the union
of two sets of $2^{d-1}$ copies of $\Lambda_n^*$ plus some margin
space. We first consider $t=\frac{1}2$. We can lower bound
\begin{eqnarray*}
&&Z_{\Lambda_{n+1}^*} \bigl(\beta, 2^{(n+1)d} \bigl(L_0^*+R'
\bigr)^d \bigl(\vectt{\rho }+\vectt{\rho}'\bigr)/2
\bigr)
\\
&&\qquad\geq \bigl(Z_{\Lambda_{n}^*} \bigl(\beta, 2^{nd} \bigl(L_0^*+R'
\bigr)^d \vectt{\rho } \bigr) \bigr)^{2^{d-1}}
\bigl(Z_{\Lambda_{n}^*} \bigl(\beta, 2^{nd} \bigl(L_0^*+R'
\bigr)^d \vectt{\rho }' \bigr) \bigr)^{2^{d-1}}.
\end{eqnarray*}
We divide by $|\Lambda_{n+1}^*|$ and pass to the limit, and this gives
equation~\eqref{eq:etaconv}
for the case $t=\frac{1}2$. The general case is obtained by iterating
the inequality.
\end{pf}

The following is a technical preparation for the proof of the local
boundedness of $\eta_j(\beta,\cdot)$ in Lemma~\ref{lem:nonempty-int}
and will also be used later. We define a cluster partition function
with volume constraint: for $a,\beta>0$, $k\in\N$, let
%
%
\begin{eqnarray}
\label{Zcldef} &&Z_k^{\cl,a}(\beta):= \frac{1}{k!a^d} \int
_{ ([0,a]^d)^{k}} \e^{-\beta
U(x_1,x_2,\ldots,x_k)} \mathbf{1} \bigl\{
\{x_1,x_2,\ldots,x_k\}
\nonumber
\\[-8pt]
\\[-8pt]
\nonumber
&&\hspace*{216pt}\mathrm{connected} \bigr
\} \,\dd x_1\cdots\dd x_k.
\end{eqnarray}

%
\begin{lemma} \label{lem:zka-lb}
Let $\delta\in(0, [R-r_\mathrm{hc}]/3)$. There is a $C(\delta)\in\R
$ such that for all
$k\in\N$ and $a_k > \delta+ k^{1/d}(r_\mathrm{hc} + 2\delta)$,
%
%
\begin{equation}
\label{eq:zka-lb} a_k^d Z_k^{\cl,a_k}(
\beta) \geq\bigl|B(0,\delta/2)\bigr|^k \exp\bigl( - \beta C(\delta) k\bigr).
\end{equation}
\end{lemma}

\begin{pf} The cube $[0,a_k]^d$ is large enough so that, for some $h
\in( r_\mathrm{hc} + 2 \delta, R-\delta)$ and some ${\theta} \in\R^d$,
the cubic lattice $[0,a_k]^d \cap ({\theta}+ (h \Z)^d)$
contains at least $k$ points all having distance $\geq\delta/2$ to
the boundary of the box.
By placing particles in the lattice, we obtain an $(R-\delta
)$-connected reference configuration $(x_1,\ldots, x_k)\in([0,a_k]^d)^k$
with the following properties:
\begin{itemize}
\item all points have distance $\geq\delta/2$ to the boundary of
$[0,a_k]^d$;
\item distinct points $x_i,x_j$ have distance $> r_\mathrm{hc} + \delta$
to each other.
\end{itemize}
We can
lower bound $Z_k^{\cl,a_k}(\beta)$ by integrating only over those
configurations with exactly one particle per ball
$B(x_i,\delta/2)$. Such a configuration is always $R$-connected.
Moreover the energy of such a configuration can be upper bounded by
$C(\delta) k$
with
\[
C(\delta):= \sum_{{\ell} \in\Z^d\setminus\{0\}} \sup_{s \in(r_\mathrm
{hc}+\delta,R)}
\bigl|v \bigl(s |{\ell}| \bigr)\bigr | <\infty,
\]
and equation~\eqref{eq:zka-lb} follows.
\end{pf}

%
\begin{lemma}[$\overline{\{\eta_j(\beta,\cdot)<\infty\}}$ has nonempty
interior]\label{lem:nonempty-int}For $\overline\rho\in(0,\infty)$, let
\[
A_j(\overline\rho):= \Biggl\{ (\rho, \rho_1,\ldots,
\rho_j) \in(0,\infty )\times[0,\infty)^j \Big| \rho\leq
\overline\rho, \sum_{k=1}^{j} k
\rho_k \leq\rho \Biggr\}.
\]
Let $\delta\in(0, (R- r_\mathrm{hc})/3)$ and $C(\delta)$ be as in
Lemma~\ref{lem:zka-lb}.
Fix $\overline\rho(\delta):= (r_\mathrm{hc}+R+ 2 \delta)^{-d}$.
Then for all $\vectt{\rho} \in A_j(\overline\rho(\delta)) \cap\mathcal
{D}_j$, we have
$\eta_j(\beta,\vectt{\rho}) \leq C(\delta) - \beta^{-1} \log|B(0,\break\delta
/2)|<\infty$. In particular,
\[
A_j\bigl(\overline\rho(\delta)\bigr) \cap\mathcal{D}_j
\subset\bigl\{\eta_j(\beta,\cdot)<\infty\bigr\}.
\]
\end{lemma}

\begin{pf}
We first give an appropriate lower bound for the constrained partition
function for the two extreme cases when (1) all clusters have the same
size $k\in\{1,\ldots,j\}$, and (2) all clusters are larger than $j$.
Afterwards, we use the convexity of $\eta_j(\beta,\cdot)$ (see
Lemma~\ref{lem:etaconv}) to handle all other cases.

Thus fix $\vectt{\rho}= (\rho,\rho_1,\ldots,\rho_j) \in\mathcal{D}_j
\cap A_j(\overline\rho(\delta))$. In the first case, let $k\in\{1,\ldots,j\}$ and $\vectt{\rho}=\vectt{\rho}^{\ssup k}$ with $\vectt{\rho}^{\ssup
k}_k=\rho_k=\rho/k$ and $\vectt{\rho}^{\ssup k}_i=\rho_i=0$ for $i\neq
k$. It follows that the $N^\ssup{n}$, $N_i^\ssup{n}$'s defined as in
equation~\eqref{eq:special-ns}
satisfy $N^\ssup{n} = k N_k^\ssup{n}$ and $N_i^\ssup{n}=0$ for $i\neq
k$. Furthermore, let $a_k > \delta+ k^{1/d}(r_\mathrm{hc} + 2\delta)$
such that $\rho(a_k + R)^d < k$. We are going to use the boxes $\Lambda
_n^*$ defined above.
In $\Lambda_n^*$, we place
cubes of side-length $a_k$ with mutual distance $\geq R$. As $n\to
\infty$, the number of such boxes
behaves like
\[
\ell_n:= \biggl\lfloor\frac{ |\Lambda_n^*|}{ (a_k+R)^d} \biggr\rfloor\sim
\frac{N^\ssup{n} / \rho}{ (a_k +R)^d} > \frac{N^\ssup{n}}{k}.
\]
Thus we can lower bound the partition function by requiring that each
$k$-cluster lies entirely
in one of the above boxes, and there is at most one cluster in each
such box. This gives
%
%
\begin{eqnarray}
\label{eq:nka} Z_{\Lambda_n^*}\bigl(\beta,N^{\ssup n},N_1^{\ssup n},
\ldots,N_j^{\ssup n} \bigr) &\geq&\pmatrix{\ell_n
\cr
N^{\ssup n}/k} \bigl(a_k^d Z_k^{\cl,a_k}(
\beta) \bigr)^{N^{\ssup n}/k}
\nonumber
\\[-8pt]
\\[-8pt]
\nonumber
&\geq&\bigl|B(0,\delta/2)\bigr|^{N^{\ssup n}} \exp\bigl( - \beta
N^{\ssup n} C(\delta)\bigr),
\end{eqnarray}
where in the last step we used Lemma~\ref{lem:zka-lb} and estimated the
counting term against one. Thus we find
\[
\lim_{n\to\infty} \frac{1}{|\Lambda_n^*|} \log Z_{\Lambda_n^*}\bigl(
\beta,N^{\ssup n},N_1^{\ssup n},\ldots,N_j^{\ssup n}
\bigr) \geq\rho \bigl(- \beta C(\delta) + \log\bigl|B(0,\delta/2)\bigr| \bigr).
\]
Thus
\[
\eta_j\bigl(\beta,\vectt{\rho}^{\ssup k}\bigr) \leq\rho \bigl(
C(\delta) - \beta ^{-1} \log\bigl|B(0,\delta/2)\bigr| \bigr).
\]

In the next step, we assume that $\vectt{\rho}=\vectt{\rho}^{\ssup0}$
with $\vectt{\rho}^{\ssup0}_k=\rho_k =0$ for all $k=1,\ldots,j$. Again,
we define $N^{\ssup n}$ and the $N_i^{\ssup n}$ by \eqref
{eq:special-ns}. We now lower bound the constrained partition function
by putting all particles into one cluster:
\[
Z_{\Lambda_n^*} \bigl(\beta,N^\ssup{n},N_1^{\ssup n},
\ldots,N_j^{\ssup n}\bigr) \geq\bigl|\Lambda_n^*\bigr|
Z_{N^\ssup{n}}^{\cl,L_n^*}(\beta) \qquad\mbox{for } N^\ssup{n} \geq j+1.
\]
Observe that $a_n:=L_n^*$ satisfies the conditions from Lemma~\ref
{lem:zka-lb}, and thus we also have
\[
\eta_j\bigl(\beta,\vectt{\rho}^{\ssup0}\bigr) \leq\rho \bigl(
C(\delta) - \beta ^{-1} \log\bigl|B(0,\delta/2)\bigr| \bigr).
\]

In the general case, let $q_k:= k \rho_k /\rho$ for $k\in\{1,\ldots,j\}
$ and $q_{0}:=1- \sum_{k=1}^j q_k$. Then $q_0, q_1,\ldots,q_j\geq0$ are
dyadic fractions and satisfy $\sum_{k=0}^{j} q_k =1$. Furthermore,
$\vectt{\rho}=\sum_{k=0}^{j} q_k\vectt{\rho}^{\ssup k}$. It follows from
Lemma~\ref{lem:etaconv} that
\[
\eta_j(\beta,\vectt{\rho}) \leq\sum_{k=0}^j
q_k \eta_j\bigl(\beta,\vectt{\rho }^{\ssup k}\bigr)
\leq\rho \bigl( C(\delta) - \beta^{-1} \log\bigl|B(0,\delta /2)\bigr| \bigr).
\]
\upqed\end{pf}

\subsubsection{Extension of \texorpdfstring{$\eta_j(\beta,\cdot)$}{$eta_j(beta,cdot)$} to \texorpdfstring{$\mathbb{R}^{j+1}$}{$\mathbb{R}^{j+1}$}}

We now extend $\eta_j(\beta,\cdot)\dvtx\mathcal{D}_j\to\R
\cup\{\infty\}$
to a convex, lower semi-continuous function $f_j(\beta,\cdot)\dvtx\R
^{j+1} \to
\R\cup\{\infty\}$. We follow the proof of \cite{Ru99}, Proposition~3.3.4,
page~45.
Let $\Gamma_j$ be the closure of $\{\eta_j(\beta,\cdot)<\infty\}$, and
let $\Delta_j$ be the interior of $\Gamma_j$. Note that $\Gamma_j\subset
[0,\infty)^{j+1}$, as $\eta_j(\beta,\cdot)=\infty$ on $\R
^{j+1}\setminus[0,\infty)^{j+1}$.

%
\begin{lemma}\textup{(1)} The interior $\Delta_j$ of $\Gamma_j$ is nonempty.\vspace*{-6pt}
\begin{longlist}[(2)]
\item[(2)] The restriction of $\eta_j(\beta,\cdot)$ to $\mathcal{D}_j \cap
\Delta_j$
has a unique continuous extension $\widetilde f_j(\beta,\cdot)\dvtx
\Delta_j \to\R$.
\item[(3)] Define $f_j(\beta,\cdot)\dvtx\R^{j+1} \to\R\cup\{\infty\}$ by
%
%
\begin{equation}
\label{eq:fjboundary} f_j(\beta,\vectt{\rho})= %
\cases{ \widetilde
f_j(\beta,\vectt{\rho}),&\quad $\mbox{if } \vectt{\rho}\in\Delta
_j,$\vspace*{2pt}
\cr
+\infty,&\quad$\mbox{if } \vectt{\rho}\in\overline
\Delta_j^{\mathrm c},$\vspace *{2pt}
\cr
\displaystyle\mathop{\liminf_{\vectt{\rho}'\to\vectt{\rho} }}_ {\vectt{\rho}' \in\Delta
_j}f_j\bigl(\beta,\vectt{\rho}'
\bigr),&\quad$\mbox{if }\vectt{\rho}\in\partial\Delta_j$.}
\end{equation}
Then $f_j(\beta,\cdot)$
is convex and lower semi-continuous, and
%
%
\begin{equation}
\label{eq:fjradial} f_j(\beta,\vectt{\rho}) = \lim_{t\downarrow0}
f_j \bigl(\beta,\vectt{\rho } + t \bigl(\vectt{\rho}' -
\vectt{\rho}\bigr) \bigr),\qquad \vectt{\rho} \in\partial\Delta_j, \vectt{
\rho}' \in\Delta_j.
\end{equation}
\item[(4)]
%
%
\begin{eqnarray}
\label{eq:fjdom-lem} \qquad&&\bigl\{f_j(\beta,\cdot)<\infty\bigr\}
\nonumber
\\[-8pt]
\\[-8pt]
\nonumber
&&\qquad \subset
\overline{\Delta}_j \subset \Biggl\{ (\rho,\rho_1,\ldots,
\rho_j) \in[0,\infty)^{j+1} | \rho\in[0,\rho_\mathrm{cp}],
\sum_{k=1}^j k \rho_k \leq
\rho \Biggr\}.
\end{eqnarray}
\end{longlist}
\end{lemma}

\begin{pf}
(1) This follows from Lemma~\ref{lem:nonempty-int}.

(2) For the existence and uniqueness of a continuous extension in
$\Delta_j$, follow~\cite{Ru99}, page~45.
The key point is that in $\Delta_j$, $\eta_j(\beta,\cdot)$ is a
locally uniformly bounded, densely
defined, convex function in the sense of Lemma~\ref{lem:etaconv}.

(3) Let us extend $\widetilde f_j(\beta,\cdot)$ to $\R^{j+1}$ with
$\widetilde f_j(\beta,\vectt{\rho}) = \infty$ for $\vectt{\rho} \in\R
^{j+1}\setminus\Delta_j$. Then $\widetilde f_j(\beta,\cdot)$ is
convex, but may fail to be lower semi-continuous.
Furthermore, $\widetilde f_j(\beta,\cdot)$ and $f_j(\beta,\cdot)$ can
differ only on $\partial\Delta_j$.
The lower semi-continuous hull of $\widetilde f_j(\beta,\cdot)$ is
\[
\cl \widetilde f_j(\beta,\vectt{\rho}):= \liminf
_{\vectt{\rho}' \to
\vectt{\rho}} \widetilde f_j\bigl(\beta,\vectt{
\rho}'\bigr), \qquad\vectt{\rho} \in\R^{j+1};
\]
see \cite{hl01}, Definition~1.2.4, page~79. This is a convex, lower
semi-continuous function which
coincides with $\widetilde f_j(\beta,\vectt{\rho})$ in $\Delta_j$ \cite{hl01},
Proposition~1.2.6, page~80.
It follows that $\cl \widetilde f_j(\beta,\vectt{\rho})$ coincides with
$f_j(\beta,\cdot)$
in $\Delta_j$. It is elementary to see that in the definition of $\cl
\widetilde f_j(\beta,\cdot)$, the limit inferior can be restricted
to those $\vectt{\rho}' \to\vectt{\rho}$ that are in $\Delta_j$. In
other words, $\cl \widetilde f_j(\beta,\cdot)$ and $f_j(\beta,\cdot)$
coincide throughout $\R^{j+1}$. This shows that $f_j(\beta,\cdot)$ is
convex and lower semicontinuous. Equation~\eqref{eq:fjradial} follows
from \cite{hl01}, Proposition~1.2.5.

(4) The first inclusion follows from the definition of $f_j(\beta,\cdot
)$, and the second from \eqref{eq:etadom}.
\end{pf}

\subsubsection{Limit behavior along general sequences}

%
\begin{lemma}\label{lem-Delta}
Fix $(\rho,\rho_1,\ldots,\rho_j) \in(0,\infty)^{j+1}$.
Let $(N_1^\ssup{N})_{N\in\N},\ldots,\break   (N_j^\ssup{N})_{N\in\N}$
be $\N_0$-valued sequences and $(\Lambda_N)_{N\in\N}$ a sequence of cubes
such that as $N \to\infty$, \eqref{convN1bisj} holds.
Then, if $(\rho,\rho_1,\ldots,\rho_j)$ is in $\Delta_j$,
%
\begin{eqnarray}
\label{rhoinDelta} &&\lim_{N\to\infty} \frac{1}{|\Lambda_N|} \log
Z_{\Lambda_N}\bigl(\beta, N,N_1^\ssup{N},\ldots,
N_j^\ssup{N}\bigr)
\nonumber
\\[-8pt]
\\[-8pt]
\nonumber
&&\qquad= - \beta f_j(\beta,\rho,
\rho_1,\ldots,\rho_j) \in\R.
\end{eqnarray}
\end{lemma}

\begin{pf} We proceed as in~\cite{Ru99}, page~47. We first prove the
lower bound in~\eqref{rhoinDelta}.
We will approximate $(\rho,\rho_1,\ldots,\rho_j)$ with $(\rho^*,\rho
^*_1,\ldots,\rho^*_j) \in\mathcal{D}_j$ satisfying $\rho>\rho^*$ and
$\rho^*_1 \leq\rho_1, \ldots,\rho^*_j \leq\rho_j$. The idea is to pick
the size parameter $n=n(N)\to\infty$ of the special sequence of cubes
$\Lambda_{n(N)}^*$ introduced at the beginning of Section~\ref
{sec-specialcubes} in such a way that the cubes are small compared to
$\Lambda_N$.
Hence, we can place a lot of them inside $\Lambda_N$ at mutual distance
$\geq R$. Afterwards, we distribute the particles and clusters
inside a certain number of special cubes according to the
distribution $(\rho^*,\rho^*_1,\ldots, \rho^*_j)$ and place
the few remaining particles somewhere else in $\Lambda_N$.

Let $(n(N))_{N\in\N}$ be an integer-valued sequence such that
\[
n(N) \to\infty\quad\mbox{and}\quad\bigl |\Lambda_{n(N)}^*\bigr|^2 / |\Lambda
_N| \to0.
\]
We define $N_*^{\ssup{n(N)}}$ and $N_{*,k}^{\ssup{n(N)}}$ by \eqref
{eq:special-ns} with $n$ replaced by $n(N)$ and $\rho,\rho_1,\ldots,\rho
_j$ replaced by $\rho^*,\rho^*_1,\ldots,\rho^*_j$.
Let $m_N \in\N_0$ and $r^\ssup{N} \in\{0,\ldots,\break N_*^{\ssup{n(N)}} -
1\}$ be such that
\[
N = m_N N_*^{\ssup{n(N)}} + r^\ssup{N}.
\]
This is possible because $\rho>\rho^*$ and therefore $N>N_*^{\ssup
{n(N)}}$ for all sufficiently large $N$.
For $k\in\{1,\ldots,j\}$, define $r_k^\ssup{N}$ by
\[
N_k^\ssup{N} = m_N N_{*,k}^{\ssup{n(N)}}
+ r_k^\ssup{n(N)}.
\]

We claim that, for sufficiently large $N$, the $r_k^\ssup{N}$ are
nonnegative integers. Indeed,
this follows from
\[
N_k^\ssup{N} \sim\rho_k |\Lambda_N|
\quad\mbox{and}\quad m_N N_{*,k}^{\ssup{n(N)}} \sim
\frac{\rho|\Lambda_N|}{\rho^* |\Lambda
_{n(N)}^*|} \rho_k^* \bigl|\Lambda_{n(N)}^*\bigr| = \frac\rho{
\rho^*} \rho_k^* |\Lambda_N|
\]
in combination with $\rho_k \geq\rho_k^* >\frac\rho{\rho^*} \rho_k^* $.
Moreover, we can place $m_N + r^\ssup{N}$ copies of $\Lambda_{n(N)}^*$
with mutual distance $\geq R$ inside $\Lambda_N$. This is so because
\begin{eqnarray*}
m_N\bigl|\Lambda_{n(N)}^*\bigr| &\sim&\frac{\rho}{\rho^*} |
\Lambda_N|\quad \mbox {and} \\
 r^\ssup{N} \bigl|\Lambda_{n(N)}^*\bigr|
&=& O\bigl( N_*^{\ssup{n(N)}}\bigl|\Lambda_{n(N)}^*\bigr|\bigr) = O\bigl( \rho^* |
\Lambda_{n(N)}|^2 \bigr) = o\bigl(|\Lambda_N|\bigr).
\end{eqnarray*}
We lower bound the constrained partition function with parameters
$N,N_1^\ssup{N},\break \ldots,   N_j^\ssup{N}$ by distributing
first particles and clusters in the $m_N$ boxes following the distribution
$N^{\ssup{n(N)}}_{*,k}$. This leaves $r^\ssup{N}$ particles. Of those
we distribute
first $k r_k^\ssup{N}$ as clusters of size $k$, one per special cube,
and then we distribute the remaining $s^\ssup{N}$ particles into clusters
of size $j+1$ except maybe for one of size between $j+2$ and $2j+1$.
Pretend for simplicity that they all have size $j+1$. Then we get
\begin{eqnarray*}
&&\log Z_{\Lambda_N}\bigl(\beta,N,N_1^\ssup{N},
\ldots,N_j^\ssup{N}\bigr) \\
&&\qquad \geq m_N \log
Z_{\Lambda_{n(N)}^*} \bigl(\beta,N_*^{\ssup{n(N)}}, N_{*,1}^{\ssup{n(N)}},
\ldots,N_{*,j}^{\ssup{n(N)}}\bigr)
\\
& &\qquad\quad{}+ \sum_{k=1}^{j+1} r^\ssup{N}_k
\log Z_k^{\cl,L_{n(N)}^*} (\beta),
\end{eqnarray*}
where $L_{n(N)}^*$ denotes the side length of $\Lambda_{n(N)}^*$. Using that
$\sum_{k=1}^{j+1} r_k^\ssup{N} \leq r^\ssup{N} \leq N_*^{\ssup{n(N)}}
= o(|\Lambda_N|)$,
we get
\[
\liminf_{N\to\infty} \frac{1}{|\Lambda_N|} \log Z_{\Lambda_N}\bigl(
\beta,N,N_1^\ssup{N},\ldots,N_j^\ssup{N}
\bigr) \geq- \beta\frac{\rho}{\rho^*} f_j\bigl(\beta,\rho^*,
\rho_1^*,\ldots, \rho_j^*\bigr).
\]
Now let $(\rho^*,\rho^*_1,\ldots,\rho^*_j) \to(\rho,\rho_1,\ldots,\rho_j)$
and use the continuity of $f_j(\beta,\cdot)$ in $\Delta_j$, to obtain
\[
\liminf_{N\to\infty} \frac{1}{|\Lambda_N|} \log Z_{\Lambda_N}\bigl(
\beta,N,N_1^\ssup{N},\ldots,N_j^\ssup{N}
\bigr) \geq- \beta f_j(\beta,\rho,\rho_1,\ldots,
\rho_j).
\]

Now we prove the upper bound in \eqref{rhoinDelta}.
First of all, let us observe that the lower bound holds not only for
sequences of cubes, but more generally for sequences of domains
$\Lambda''_N$ that converge to infinity in the Fisher sense,
as can be shown along the lines of our proof and \cite{Ru99}.
We shall need the statement not for general Fisher domains but only
for $\Lambda''_N$ defined below,
which is an L-shaped domain that is a difference of two cubes.

Now fix $C \in(0,\frac{1}2)$. For $N\in\N$, let $n(N)\in\N$ be so
large that $\Lambda_{n(N)}^*$
contains $\Lambda_N$ and satisfies
\[
0< C \leq\frac{|\Lambda_N|}{|\Lambda_{n(N)}^*|} \leq\frac
{1}{2},\qquad n\in\N.
\]
Let $\Lambda''_{N}$ be the set of points in $\Lambda_{n(N)}^*$
having distance $> R'$ to $\Lambda_N$. Then $(|\Lambda_N| + |\Lambda
''_N| )/|\Lambda_{n(N)}^*|
\to1$.
Let $\vectt{\rho}^*=(\rho^*,\rho^*_1,\ldots,\rho^*_j) \in\Delta_j \cap
\mathcal{D}_j$ such that
$\rho^*_k >0$. Define $N_*^{\ssup{n(N)}}$ and $N_{*,k}^{\ssup{n(N)} }$
as in equation~\eqref{eq:special-ns} with $n$ replaced by $n(N)$ and
$\rho,\rho_1,\ldots,\rho_j$ replaced by $\rho^*,\rho^*_1,\ldots,\rho
^*_j$. Then
%
%
\begin{eqnarray}
\label{ineq:lbaux} %
\qquad&& Z_{\Lambda_{n(N)}^*}\bigl(
\beta,N_*^{\ssup{n(N)}},N_{*,1}^{\ssup
{n(N)}},\ldots,N_{*,j}^{\ssup{n(N)}}
\bigr)
\nonumber
\\
&&\qquad\geq Z_{\Lambda_{N}}\bigl(\beta,N,N_1^\ssup{N},
\ldots,N_j^\ssup{N}\bigr)\\
&&\qquad\quad{} \times Z_{\Lambda_N''}\bigl(
\beta,N_*^{\ssup{ n(N)}} - N,N_{*,1}^{\ssup{n(N)}}-
N_1^\ssup{N}, \ldots,N_{*,j}^{\ssup{n(N)}} -
N_j^\ssup{N} \bigr). \nonumber
\end{eqnarray}
Assume for simplicity that $|\Lambda_N| / |\Lambda_{n(N)}^*| \to\alpha
\in(0,1/2]$
(otherwise go to suitable subsequences). Then
\[
\frac{N_*^{\ssup{n(N)}} - N}{|\Lambda_N''|} \sim\frac{ \rho^* |\Lambda_{n(N)}^*| - \rho|\Lambda_N|}{|\Lambda_N''|} \to\frac{\rho^* - \rho \alpha}{1-\alpha} =:
\rho''.
\]
Define $\rho''_1,\ldots,\rho''_j$ in an analogous way, and put $\vectt{\rho}''=(\rho'',\rho''_1,\ldots,\rho''_j)$. Thus
$\vectt{\rho}^* = \alpha\vectt{\rho} + (1-\alpha) \vectt{\rho}''$ and
\[
\bigl|\vectt{\rho}'' - \vectt{\rho}\bigr| = (1-
\alpha)^{-1} \bigl|\vectt{\rho}- \vectt{\rho }^*\bigr| \leq2 \bigl|\vectt{\rho}-
\vectt{\rho}^*\bigr|,
\]
with $|\cdot|$ the Euclidean norm.
Let $\eps>0$ such that $B_\eps(\vectt{\rho}) \subset\Delta_j$. Now
additionally assume that $\vectt{\rho}^* \in B_{\eps/2}(\vectt{\rho})$.
Thus $\vectt{\rho}''\in\Delta_j$.
In equation~\eqref{ineq:lbaux}, we take logarithms, divide by $|\Lambda
_{n(N)}^*|$ and pass to the limit $N\to\infty$, which gives
\begin{eqnarray*}
&&- \beta f_j\bigl(\beta,\vectt{\rho}^*\bigr)
\\
&&\qquad\geq\alpha\limsup
_{N\to\infty} \frac{1}{|\Lambda_N|} \log Z_{\Lambda
_{N}}\bigl(
\beta,N,N_1^\ssup{N},\ldots,N_j^\ssup{N}
\bigr) - (1-\alpha) \beta f_j \bigl(\beta,\vectt{\rho}''
\bigr).
\end{eqnarray*}
To conclude we let $\vectt{\rho}^*\to\vectt{\rho}$ (hence $\vectt{\rho
}''\to\vectt{\rho}$) and use $\alpha>0$ and the continuity
of $f_j(\beta,\cdot)$ at $\vectt{\rho}$.
\end{pf}

%
\begin{lemma}\label{lem-Deltaclcomp}
Assume the situation of Lemma~\ref{lem-Delta}. If $\vectt{\rho}=(\rho,\rho_1,\ldots,\rho_j)$ is in $\overline\Delta_j^{\mathrm c}$ or in
$\partial\Delta_j$, then
\[
\limsup_{N\to\infty} \frac{1}{|\Lambda_N|} \log Z_{\Lambda_N}\bigl(
\beta, N,N_1^\ssup{N},\ldots, N_j^\ssup{N}
\bigr) \leq- \beta f_j(\beta,\vectt{\rho}).
\]
[Recall that $f_j(\beta,\vectt{\rho})=\infty$ in the first case.]
\end{lemma}

\begin{pf} We proceed as in \cite{Ru99}, Proposition~3.3.8, page~48.
One can show that
there is an $\alpha\in(0,1/2]$ such that for $\vectt{\rho}^* \in
\mathcal{D}_j$ satisfying $\rho^*_k>0$ whenever
$\rho_k>0$, and $\vectt{\rho}'' \in\Delta_j$ satisfying
%
%
\begin{equation}
\label{eq:alpharhos} \vectt{\rho}^* = \alpha\vectt{\rho} + (1-\alpha) \vectt{
\rho}'',
\end{equation}
it holds that
%
%
\begin{eqnarray}
\label{eq:bound-boundary} - \beta\eta_j\bigl(\beta,\vectt{\rho}^*\bigr) &\geq&
\alpha\limsup_{N\to\infty} \frac{1}{|\Lambda_N|} \log Z_{\Lambda_{N}}
\bigl(\beta,N,N_1^\ssup{N},\ldots,N_j^\ssup{N}
\bigr)
\nonumber
\\[-8pt]
\\[-8pt]
\nonumber
&&{} - (1-\alpha) \beta f_j\bigl(\beta, \vectt{\rho}''
\bigr).
\end{eqnarray}
The proof of this is similar to the proof of the upper bound in
Lemma~\ref{lem-Delta}.

(a) Consider the case $\vectt{\rho} \in\overline\Delta_j^{\mathrm c}$.
For $\vectt{\rho}''\in\Delta_j$,
we define $\vectt{\rho}^*$ by \eqref{eq:alpharhos}. By choosing $\vectt
{\rho}''$ close enough
to $\partial\Delta_j$,
we can ensure that $\vectt{\rho}^* \in\mathcal{D}_j \cap\overline
\Delta_j ^\mathrm{c}$.
Thus we conclude from~\eqref{eq:bound-boundary} that
\[
\limsup_{N\to\infty} \frac{1}{|\Lambda_N|} \log Z_{\Lambda_N}\bigl(
\beta, N,N_1^\ssup{N},\ldots, N_j^\ssup{N}
\bigr) = - \infty.
\]

(b) If $\vectt{\rho} \in\partial\Delta_j$, let $\vectt{\rho}'(\eps)
\in\Delta_j \cap B_\eps(\vectt{\rho})$
be such that $f_j(\beta,\vectt{\rho}'(\eps)) \to f_j(\beta,\vectt{\rho
})$ as $\eps\downarrow0$.
By \cite{hl01}, Lemma 2.1.6, page~35,
the half-open line segment $(\vectt{\rho},\vectt{\rho}'(\eps)]$ is
contained in $\Delta_j$.
Since $\mathcal{D}_j$ is dense and because of the continuity of
$f_j(\beta,\cdot)$ at
$\vectt{\rho}'(\eps)$, we can find $\vectt{\rho}''(\eps) \in\Delta_j
\cap B_\eps(\vectt{\rho})$
such that:
\begin{itemize}
\item$\vectt{\rho}^*(\eps)$, defined by \eqref{eq:alpharhos} with $\vectt
{\rho}''$ replaced by $\vectt{\rho}''(\eps)$, is in $\Delta_j\cap
\mathcal{D}_j \cap B_\eps(\vectt{\rho})$;
\item$|f_j(\beta,\vectt{\rho}'(\eps)) - f_j(\beta,\vectt{\rho}''(\eps))|
\leq\eps$,
so that $f_j(\beta,\vectt{\rho}''(\eps)) \to f_j(\beta,\vectt{\rho})$ as
$\eps\to0$.
\end{itemize}
It follows from equation~\eqref{eq:bound-boundary} that
\begin{eqnarray*}
&&\alpha\limsup_{N\to\infty} \frac{1}{|\Lambda_N|} \log Z_{\Lambda_N}
\bigl(\beta, N,N_1^\ssup{N},\ldots, N_j^\ssup{N}
\bigr) \\
&&\qquad \leq\limsup_{\eps\to0} \bigl( - \beta f_j\bigl(
\beta,\vectt{\rho}^*(\eps )\bigr) + (1-\alpha) \beta f_j\bigl(\beta,
\vectt{\rho}'' (\eps)\bigr) \bigr)
\\
&&\qquad \leq- \alpha\beta f_j(\beta,\vectt{\rho}).
\end{eqnarray*}
\upqed\end{pf}

\begin{pf*}{Proof of Proposition~\ref{prop:constrained-free-energy}}
This is now straightforward from the previous lemmas.
\end{pf*}

\subsection{The \texorpdfstring{$\rho$}{$rho$}-sections of \texorpdfstring{${\Delta_j}$}{$Delta_j$}}\label{sec-Deltaj}

We already saw that the set $\{f_j(\beta,\cdot)<\infty\}$ has
nonempty interior $\Delta_j$.
In view of the large deviations principle we are interested in
properties of the map $(\rho_1,\ldots,\rho_j) \mapsto f_j(\beta,\rho,\rho_1,\ldots,\rho_j)$ at fixed $\beta$ and $\rho$.
This means that we look at the restriction of $f_j(\beta,\cdot)$ to the
hyperplane of constant density $\rho$.

Now, this restricted map inherits the convexity and lower semi-continuity
from $f_j(\beta,\cdot)$. The question whether the set where it is finite
has nonempty interior is, however more subtle. Closely related is the question
whether $\Delta_j$ has nonempty intersection with the hyperplane of constant
density $\rho$.

To this aim consider the $\rho$-section of $\Delta_j$,
%
%
\begin{equation}
\label{eq:cj} C_j(\rho):= \bigl\{ (\rho_1,\ldots,
\rho_j) \in(0,\infty)^j | (\rho,\rho_1,\ldots,
\rho _j) \in\Delta_j \bigr\}.
\end{equation}
Put differently, $\{\rho\}\times C_j(\rho)$ is the intersection of
$\Delta_j$
with the hyperplane of constant density $\rho$.
%
The hyperplane always cuts through the interior of $\Delta_j$, that is,
cannot be tangent to $\Delta_j$:

%
\begin{lemma} \label{lem:cj-nonempty}
For any $\rho\in(0,\rho_{\mathrm{ cp}})$, the set $C_j(\rho)$ is nonempty,
convex and open.
Moreover,
%
%
\begin{equation}
\label{eq:secclos} \overline{C_j(\rho)} = \bigl\{ (\rho_1,
\ldots,\rho_j) \in[0,\infty)^j \mid (\rho,
\rho_1,\ldots, \rho_j) \in\overline{\Delta}_j
\bigr\}.
\end{equation}
\end{lemma}

This last equation says that it does not matter whether we take first
the $\rho$-section and then close the set, or
if we close first and then take the section.

The essential ingredients of the proof of Lemma~\ref{lem:cj-nonempty}
are the convexity of $f_j(\beta,\cdot)$,
Lemma~\ref{lem:nonempty-int} and the following.

%
\begin{lemma}
\label{lem:fj-origin}
Let $\rho\in(0,\rho_\mathrm{cp})$. Then there is at least one point
$(\rho_1,\ldots,\break \rho_j) \in[0,\infty)^j$
such that $f_j(\beta,\rho,\rho_1,\ldots,\rho_j) <\infty$.
\end{lemma}

\begin{pf}
Let $N/|\Lambda_N| \to\rho$. Let $(N_1^\ssup{N},\ldots,N_j^\ssup{N})$
be such that
\[
Z_{\Lambda_N}\bigl(\beta,N,N_1^\ssup{N},
\ldots,N_j^\ssup{N}\bigr) = \max_{(N_1,\ldots,N_j) \in\N_0^j}
Z_{\Lambda_N}(\beta,N,N_1,\ldots,N_j).
\]
According to the Hardy--Ramanujan formula, the number of partitions of
$N$ is not larger than $\exp(O(\sqrt{N}))$. Thus we find
\[
Z_{\Lambda_N}(\beta,N) \leq\exp\bigl(O(\sqrt{N})\bigr) Z_{\Lambda_N}
\bigl(\beta,N,N_1^\ssup{N},\ldots,N_j^\ssup{N}
\bigr).
\]
Passing to a suitable subsequence, we may assume that $N_k^\ssup
{N}/|\Lambda_N| \to\rho_k$,
$k=1,\ldots,j$, for some $(\rho_1,\ldots,\rho_j) \in[0,\infty)^j$.
The previous inequality then yields
\[
-\infty< - \beta f(\beta,\rho) \leq- \beta f_j(\beta,\rho,\rho
_1,\ldots,\rho_j).
\]
\upqed\end{pf}

\begin{pf*}{Proof of Lemma~\ref{lem:cj-nonempty}}
Let $\rho\in(0,\rho_\mathrm{cp})$. Let $\rho'\in(\rho, \rho_\mathrm
{cp})$ and
$(\rho'_1,\ldots,\break  \rho'_j) \in[0,\infty)^j$ such that $f_j(\beta,\vectt
{\rho}')<\infty$, where
$\vectt{\rho}'=(\rho',\rho'_1,\ldots,\rho'_j)$. Hence, $\vectt{\rho}'
\in\overline\Delta_j$.
Let $\overline\rho(\delta)$ and $A(\overline\rho(\delta))$ be as in
Lemma~\ref{lem:nonempty-int}.
Let $\mathcal{C} \subset[0,\infty)^{j+1}$ be the cone with apex $\vectt
{\rho}'$
and base $A(\overline\rho(\delta))$, that is, the set of convex
combinations of
points in $A(\overline\rho(\delta))$ and $\vectt{\rho}'$.
By convexity, $\mathcal{C} \subset\Delta_j$.
Looking at the $\rho$-sections of $\mathcal{C}$ we find that $C_j(\rho
)$ is not empty.

Convexity and openness of $C_j(\rho)$ are inherited from $\Delta_j$.

Now we prove \eqref{eq:secclos}. Let $H_\rho= \{\rho\} \times\R^j
\subset\R^{j+1}$ be the
hyperplane of density~$\rho$. By \cite{hl01}, Proposition~2.1.10, page~37,
\[
\overline{\Delta_j \cap H_\rho} = \overline{
\Delta}_j \cap\overline {H}_\rho.
\]
The left-hand side is identified as $\{\rho\}\times\overline{C_j(\rho
)}$ while
the right-hand side is $\{\rho\}\times A$ with $A$ the set from the
right-hand side in
equation~\eqref{eq:secclos}.
\end{pf*}

\subsection{Proof of Proposition~\texorpdfstring{\protect\ref{prop:ldp-j}}{2.2}: LDP for the
projection of \texorpdfstring{$\bolds{\rho}_\Lambda$}{$rho_Lambda$}}\label{sec-proofLDPj}

In this section, we prove the large deviations principle for
$(\rho_{1,\Lambda},\ldots,\rho_{j,\Lambda})$ under the Gibbs measure, as
formulated in Proposition~\ref{prop:ldp-j}. This is equivalent to
showing the two bounds in~\eqref{lowerbound} and \eqref{upperbound} and
the claimed properties of $I_{\beta,\rho,j}$. Observe that the
distribution of $(\rho_{1,\Lambda},\ldots,\rho_{j,\Lambda})$ under the
Gibbs measure is concentrated on the compact set $M_\rho$. Hence, the
family of these distributions is in particular exponentially tight.
Hence, it is enough to prove the upper bound in \eqref{upperbound} for
compact sets. From this, in particular the compactness of the level
sets of $I_{\beta,\rho,j}$ follows, but we will also give an
independent proof.

For the remainder of this section, we fix $\rho\in(0,\rho_{\mathrm {cp}})$.

\subsubsection{Properties of \texorpdfstring{$I_{\beta,\rho,j}$}{$I_{beta,rho,j}$}}

Recall the function $I_{\beta,\rho,j}\dvtx[0,\infty)^j\to
\R\cup\{\infty\}$ from \eqref{eq:ij} and the $\rho$-section $C_j(\rho
)$ of $\Delta_j$ from \eqref{eq:cj}. Recall from Lemma~\ref
{lem:cj-nonempty} that $C_j(\rho)$ is nonempty, open and convex.

%
\begin{lemma} \label{lem:ij-properties}
(1) $I_{\beta,\rho,j}$ is convex, and its level sets are compact.%
\begin{longlist}[(2)]
\item[(2)]$I_{\beta,\rho,j}$ is finite in $C_j(\rho)$ and
infinite in the complement of the closure of $C_j(\rho)$.
\item[(3)] For every open set $\mathcal{O} \subset[0,\infty)^j$,
%
%
\begin{equation}
\label{eq:ij-inf} \inf_{\mathcal{O}} I_{\beta,\rho,j} = %
\cases{\displaystyle \inf_{\mathcal{O}\cap C_j(\rho)}I_{\beta,\rho,j},&\quad
$\mbox{if } \mathcal{O}\cap
C_j(\rho) \neq\emptyset,$\vspace*{2pt}
\cr
\infty,&\quad$\mbox{if }
\mathcal{O}\cap C_j(\rho) = \emptyset$.} %
\end{equation}
\end{longlist}
\end{lemma}

\begin{remark}
Equation~\eqref{eq:ij-inf} will be needed in the proof of the lower
bound for the large deviations principle.
The convexity enters in a crucial way in equation~\eqref{eq:ij-inf}.
Lower semi-continuity alone would not suffice!---(3) proves
that the open set $C_j(\rho)$ is a $I_{\beta,\rho,j}$-continuity set;
see \cite{DZ98}, page 5.
\end{remark}

\begin{pf}
(1) Convexity and lower semi-continuity are immediate consequences of
the properties for $f_j(\beta,\cdot)$, since the restriction of a
convex, lower semi-continuous function to a hyperplane is
also convex and lower semi-continuous. Thus the level sets of $I_{\beta,\rho,j}$ are closed.
By equation~\eqref{eq:fjdom-lem},
\[
\{I_{\beta,\rho,j}<\infty\} \subset \Biggl\{ (\rho_1,\ldots,
\rho_j) \in [0,\infty)^j \Big| \sum
_{k=1}^j k \rho_k \leq\rho \Biggr\}.
\]
It follows that the level sets are also bounded, hence compact.

(2) If $(\rho_1,\ldots,\rho_j)$ is in $C_j(\rho)$, then $(\rho,\rho
_1,\ldots,\rho_j) \in\Delta_j$
by definition of $C_j(\rho)$, and therefore $f_j(\beta,\rho,\rho
_1,\ldots,\rho_j) <\infty$. Hence,
$I_{\beta,\rho,j}(\rho_1,\ldots,\rho_j) <\infty$.

If $(\rho_1,\ldots,\rho_j)$ is in the complement of the closure of
$C_j(\rho)$, then
by equation~\eqref{eq:secclos},
$(\rho,\rho_1,\ldots,\rho_j)$ is in the
complement of the closure of $\Delta_j$, from which $I_{\beta,\rho,j}(\rho_1,\ldots,\rho_j) = \infty$ follows.

(3) If $\mathcal{O}$ and $C_j(\rho)$ are disjoint, $I_{\beta,\rho,j} =
+\infty$
on $\mathcal{O}$ by (2). If the sets are not disjoint, we know that
%
%
\begin{equation}
\inf_{\mathcal{O}} I_{\beta,\rho,j} = \inf_{\mathcal{O}\cap\overline{C_j(\rho)}}I_{\beta,\rho,j}
\leq \inf_{\mathcal{O}\cap C_j(\rho)}I_{\beta,\rho,j},
\end{equation}
and it remains to prove the opposite inequality.
Thus let
$\vectt{\rho}=(\rho_1,\ldots,\rho_j) \in\mathcal{O}\cap\partial
C_j(\rho)$.
Let $\vectt{\rho}' \in C_j(\rho)$.
By equation~\eqref{eq:fjradial},
\[
I_{\beta,\rho,j}(\vectt{\rho}) = \lim_{t\downarrow0}
I_{\beta,\rho,j}\bigl(\vectt{\rho} + t\bigl(\vectt{\rho}'-\vectt{
\rho}\bigr)\bigr).
\]
Because $\mathcal{O}$ is open and by \cite{hl01}, Lemma 2.1.6, page~35,
for sufficiently small $t$,
$\vectt{\rho} + t(\vectt{\rho}'-\vectt{\rho}) \in\mathcal{O} \cap
C_j(\rho)$. Thus for some
suitable $t_0>0$,
\begin{eqnarray*}
I_{\beta,\rho,j}(\vectt{\rho})& =& \lim_{t\downarrow0}
I_{\beta,\rho,j}\bigl(\vectt{\rho} + t\bigl(\vectt{\rho}'-\vectt{
\rho}\bigr)\bigr) \geq\inf_{t\in(0,t_0)} I_{\beta,\rho,j}\bigl(\vectt{
\rho} + t\bigl(\vectt{\rho }'-\vectt{\rho}\bigr)\bigr)\\
& \geq&\inf
_{\mathcal{O}\cap C_j(\rho)} I_{\beta,\rho,j}.
\end{eqnarray*}
\upqed\end{pf}

\subsubsection{The two bounds in \texorpdfstring{\protect\eqref{lowerbound}}{(2.5)} and \texorpdfstring{\protect\eqref{upperbound}}{(2.6)}}

For $A \subset[0,\infty)^j$, let
\[
\mathcal{P}_{N}(j,A):= \Biggl\lbrace(N_1,
\ldots,N_j) \in\N_0^j\Big | \bigl(N_1/|
\Lambda_N|,\ldots, N_N/|\Lambda_N| \bigr) \in A,
\sum_{k=1}^j k N_k \leq N
\Biggr\rbrace.
\]
We note that the probability of finding $(\rho_{1,\Lambda_N},\ldots,
\rho_{j,\Lambda_N})$ in the set $A$ is a sum
of constrained partition functions
\begin{eqnarray*}
&&\mathbb{P}_{\beta,\Lambda_N}^{\ssup N} \bigl( (\rho_{1,\Lambda_N},\ldots,
\rho_{j,\Lambda_N}) \in A \bigr) \\
&&\qquad= \frac{1}{Z_{\Lambda_N}(\beta,N)} \sum
_{(N_1,\ldots,N_N) \in\mathcal{P}_N(j,A)} Z_{\Lambda_N}( \beta,N,N_1,
\ldots,N_j).
\end{eqnarray*}

\emph{Upper bound in \textup{\eqref{upperbound}} for compact sets.}
Let $K \subset[0,\infty)^j$ be a compact set.
Let $(N_1^\ssup{N},\ldots,N_j^\ssup{N})\in\mathcal{P}_N(j,K)$ maximize
the constrained partition function over $\mathcal{P}_{N}(j,K)$, that is,
\[
Z_{\Lambda_N}\bigl(\beta,N,N_1^\ssup{N},
\ldots,N_j^\ssup{N}\bigr) = \max_{(N_1,\ldots,N_j) \in\mathcal{P}_N(j,K)}
Z_{\Lambda_N}(\beta,N,N_1,\ldots,N_j).
\]
Then
\[
\mathbb{P}_{\beta,\Lambda_N}^\ssup{N} \bigl( (\rho_{1,\Lambda_N},
\ldots, \rho_{j,\Lambda_N}) \in K \bigr) \leq \frac{|\mathcal{P}_{N}(j,K)|} {Z_{\Lambda_N}(\beta,N)}
Z_{\Lambda
_N}\bigl(\beta,N,N_1^\ssup{N},
\ldots,N_j^\ssup{N}\bigr).
\]
Now, the cardinality of $\mathcal{P}_N(j,K)$ is smaller than
the number of partitions of $N$, and therefore not larger than $\exp
(O(\sqrt{N}))$, which is $\e^{o(N)}$.
The sequence $(N_1^\ssup{N}/|\Lambda_N|,\ldots,N_j^\ssup{N}/|\Lambda
_N|)_{N\in\N}$
takes values in the compact set $K$ and therefore, going to a subsequence,
we can assume that it converges to some $(\rho_1,\ldots,\rho_j) \in K$.
Applying Proposition~\ref{prop:constrained-free-energy} we find
\begin{eqnarray*}
\limsup_{N\to\infty} \frac{1}{|\Lambda_N|} \log Z_{\Lambda_N}\bigl(
\beta,N,N_1^\ssup{N},\ldots,N_j^\ssup{N}
\bigr)& \leq&- \beta f_j(\beta,\rho,\rho_1,\ldots,
\rho_j)\\
& \leq&- \beta\inf_{K} f_j(
\beta,\rho,\cdot).
\end{eqnarray*}
This yields the upper bound in \eqref{upperbound} for $K=\Ccal$.

\emph{Lower bound in \textup{\eqref{lowerbound}} for open sets.} Let
$\mathcal{O} \subset[0,\infty)^j$ be an open set. Let $(\rho_1,\ldots,\rho_j) \in\mathcal{O}$. We can choose $(N_1^\ssup{N},\ldots,
N_j^\ssup{N}) \in\mathcal{P}_N(j,\mathcal{O})$ so that $N_k^\ssup
{N}/|\Lambda_N| \to\rho_k$, $k=1,\ldots,j$, and have
\[
\mathbb{P}_{\beta,\Lambda_N}^\ssup{N} \bigl( (\rho_{1,\Lambda_N},
\ldots, \rho_{j,\Lambda_N}) \in\mathcal{O} \bigr) \geq \frac{1} {Z_{\Lambda_N}(\beta,N)}
Z_{\Lambda_N}\bigl(\beta,N,N_1^\ssup {N},
\ldots,N_j^\ssup{N}\bigr).
\]
If $(\rho,\rho_1,\ldots,\rho_j)$ is in $\Delta_j$ or in the complement
of the
closure of $\Delta_j$, we conclude from Proposition~\ref
{prop:constrained-free-energy}
that
\[
\liminf_{N\to\infty} \frac{1}{|\Lambda_N|} \log\mathbb{P}_{\beta,\Lambda_N}^\ssup{N}
\bigl( (\rho_{1,\Lambda_N},\ldots, \rho_{j,\Lambda_N}) \in\mathcal{O} \bigr)
\geq- I_{\beta,\rho,j}(\rho_1,\ldots,\rho_j).
\]
Thus, taking on the right-hand side the supremum over all such $(\rho
_1,\ldots,\rho_j)$, we obtain
\begin{eqnarray*}
\liminf_{N\to\infty} \frac{1}{|\Lambda_N|} \log\mathbb{P}_{\beta,\Lambda_N}^\ssup{N}
\bigl( (\rho_{1,\Lambda_N},\ldots, \rho_{j,\Lambda_N}) \in\mathcal{O} \bigr)
&\geq&- \inf_{\mathcal{O} \cap C_j(\rho)} I_{\beta,\rho,j}\\
& =& - \inf
_{\mathcal{O}} I_{\beta,\rho,j}.
\end{eqnarray*}
The last equality uses Lemma~\ref{lem:ij-properties} for the case
$\mathcal{O}\cap C_j(\rho) \neq\emptyset$, and \eqref{lowerbound} is
proved in this case.
If $\mathcal{O}$ and $C_j(\rho)$ are disjoint, then $\inf_\mathcal{O}
I_{\beta,\rho,j} = \infty$, and \eqref{lowerbound} is trivially true.
This completes the proof of Proposition~\ref{prop:ldp-j}.

\subsection{The finish: Proof of the LDP for \texorpdfstring{$(\bolds{\rho}_{\Lambda_N})_{N\in\mathbb{N}}$}{$(rho_{Lambda_N})_{Nin\mathbb{N}}$}}\label{sec-finishLDP}

The proof of Theorem~\ref{thm:ldp} follows essentially from
Proposition~\ref{prop:ldp-j} and the Dawson--G\"artner theorem, the LDP
for projective limits; see \cite{DZ98}, Theorem 4.6.1. More precisely, let
\[
I_{\beta,\rho} \bigl( (\rho_k)_{k\in\N} \bigr) = \beta
\bigl( f \bigl(\beta,\rho,(\rho_k)_{k\in\N} \bigr) - f(\beta,
\rho) \bigr)
\]
with
\[
f \bigl(\beta,\rho,(\rho_k)_{k\in\N} \bigr):= \sup
_{j \in\N} f_j(\beta,\rho,\rho_1,\ldots,
\rho_j).
\]
Consider first $I_{\beta,\rho}$ as a function from $[0,\infty)^\N$ to
$\R\cup\{\infty\}$,
and endow $[0,\infty)^\N$ with the product topology,
By the Dawson--G{\"a}rtner theorem, $I_{\beta,\rho}$ is a good rate
function and
$(\vectt{\rho}_{\Lambda_N})_{N\in\N}$ satisfies a large deviations
principle with rate function~$I_{\beta,\rho}$.

Now for all $N$, $\P_{\beta,\Lambda_N}^\ssup{N}(\vectt{\rho}_{\Lambda_N}
\in M_{\rho+ \eps}) = 1$.
Moreover, $M_{\rho+\eps}$ is closed as a subset of $[0,\infty)^\N$ in
the product topology. Thus by \cite{DZ98}, Lemma 4.1.5, we conclude that
$(\vectt{\rho}_{\Lambda_N})_{N\in\N}$ satisfies a large deviations
principle also
as an $M_{\rho+\eps}$-valued random variable in this topology.

Next, one easily sees that on $M_{\rho+\eps}$ the product topology and
the $\ell^1$ topology coincide. It follows that $(\vectt{\rho}_{\Lambda
_N})$ satisfies the LDP also in this topology with the good rate
function $I_{\beta,\rho}$.

$I_{\beta,\rho}$ is convex because it is the supremum of a family of
convex functions.

Finally, if $I_{\beta,\rho}( (\rho_k)_{k\in\N})$ is finite, then, for
all $j \in\N$, we have $f_j(\beta,\rho,\rho_1,\ldots,\break  \rho_j)<\infty$
and hence by Proposition~\ref{prop:constrained-free-energy}, $\sum_{k=1}^j k \rho_k \leq\rho$.
Letting $j \to\infty$ we obtain $\sum_{k=1}^\infty k \rho_k \leq\rho
$. This proves that $\{I_{\beta,\rho}<\infty\}$ is contained in $M_\rho$.

\section{Approximation with an ideal mixture of clusters}\label{sec-boundsf}

In this section, we compare the rate function $f(\beta,\rho,\cdot)$ defined in \eqref{eq:rf-fe} with an ideal rate function. This
rate function describes a uniform mixture of clusters that do not
interact with each other. This function has a particularly simple
shape, since the combinatorial complexity does not take care of the
excluded-volume effect, that is, different clusters do not repel each other.

One of the crucial points is a lower estimate for the combinatorial
complexity of putting a given number of clusters into a large box in a
well-separated way. For this, we need to control the free energy of
clusters that fit into some box of a certain volume. It is relatively
easy to achieve this if the radius of that box is of order of the
cardinality of the cluster, that is, under the sole condition
Assumption~\ref{assV}. This will turn out in Section~\ref{sec-proofThm14} to
be sufficient for the regime in \eqref{dilutelimit}, that is, for the
proof of Theorem~\ref{thm:main}. However, in order to handle also the
much more flexible bounds in Theorem~\ref{thm:unif}, we will have to
use boxes with \textit{volume} of order of the cluster cardinality and to
make use of Assumption~\ref{ass:maxdist}.

We consider the \textit{cluster partition function}, which is defined, for
$\beta>0$ and $k \in\N$, by
\[
Z_k^\cl(\beta) = \frac{1}{k!} \int
_{ (\R^d)^{k-1}} \e^{-\beta
U_k(0,x_2,\ldots,x_k)} \mathbf{1} \bigl\{
\{0,x_2,\ldots,x_k\} \mbox{ connected} \bigr\} \,\dd
x_2\cdots\dd x_k.
\]
Recall the cluster partition function $Z_k^{\cl,a}(\beta)$ with
restriction to $[0,a]^d$ and additional factor $a^{-d}$ introduced in
\eqref{Zcldef} above. The reader easily checks that
\[
\lim_{a\to\infty} Z_k^{\cl,a}(\beta) =
Z_k^\cl(\beta),\qquad k\in\N, \beta\in(0,\infty).
\]
We also define associated cluster-free energies per particle:
%
%
\begin{equation}
\label{freeencldef} f_k^\cl(\beta): = - \frac{1}{\beta k } \log
Z_k^\cl(\beta),\qquad f_k^{\cl,a}(\beta): = -
\frac{1}{\beta k } \log Z_k^{\cl,a}(\beta).
\end{equation}
Let
%
%
\begin{equation}
f_\infty^\cl(\beta):= \liminf_{k\to\infty}
f_k^\cl(\beta) \quad\mbox {and} \quad f_\infty^\cl(
\beta,\rho):= \limsup_{k \to\infty}f_k^{\cl,L_k}(
\beta),
\end{equation}
where $L_k$ is such that the volume of $[0,L_k]^d$ is equal to $k/\rho
$. We will see in Section~\ref{sec-finftyfinite} [see Lemma~\ref
{lem:cfe-lower-bound} and \eqref{eq:akunif-b}] that these quantities
are finite. One can actually show that
they exist as limits, but we will not need that.

Now we can state our bounds. The first one expresses the (simple) bound
that comes from dropping the excluded-volume effect. Recall
definition~\eqref{eq:fideal-def} of the ideal free energy $f^\ideal$.

%
\begin{lemma}[(Lower bound)] \label{prop:lb}
For all $\beta,\rho>0$ and $\vectt{\rho} \in M_\rho$,
%
%
\begin{equation}
f (\beta,\rho, \vectt{\rho} ) \geq f^\ideal (\beta,\rho, \vectt{\rho} ).
\end{equation}
\end{lemma}

\begin{pf}
Recall the definition~\eqref{eq:constrained-pf} of the constrained
partition functions $Z_\Lambda(\beta,N,N_1,\ldots,N_N)$. We show first that
%
%
\begin{equation}
\label{eq:constrained-pf-bound} Z_\Lambda(\beta,N,N_1,\ldots,N_N)
\leq\prod_{k=1}^N \frac{(|\Lambda|
Z_k^\cl(\beta))^{N_k}}{N_k!},
\end{equation}
for all $N,N_1,\ldots,N_N \in\N_0$ with $\sum_{k=1}^N k N_k = N$. Fix
such a vector $(N,N_1,\break \ldots,  N_N)$. Let $\vect{x}=(x_1,\ldots,x_N) \in
\Lambda^N$ with $N_1$ clusters of size $1$, $N_2$ clusters of size~$2$,
etc. Consider the graph with vertices $\{1,\ldots,N\}$ and edges those
$\{i,j\}, i \neq j$,
where $|x_i - x_j | \leq R$. The graph splits into connected
components; this induces a partition $\mathcal{I}(\vect{x})$ of the
index set $\{1,\ldots,N\}$. The set partition has $N_1$ sets of size
$1$, $N_2$ sets of size $2$, etc. Let $\mathcal{J}= \mathcal
{J}((N_k)_k)$ be the collection of such set partitions of $\{1,\ldots,N\}$. Note that the integral of $\e^{-\beta U_N}$ over $\{\vect
{x}\dvtx\mathcal{I}(\vect{x})=\mathcal{I}\}$ does not depend on
$\mathcal{I}\in\mathcal{J}$. The cardinality of $\mathcal{J}$ is
\[
|\mathcal{J}| = \frac{N!}{ \prod_{k=1}^N  (N_k! k!^{N_k} )}.
\]
Therefore, for any $\mathcal{I}^\ssup{0}\in\mathcal{J}$, we may write
\begin{eqnarray*}
Z_\Lambda(\beta,N,N_1,\ldots,N_N) &=&
\frac{1}{N!} \sum_{\mathcal{I} \in
\mathcal{J}} \int
_{\Lambda^N} \e^{-\beta U_N(\vect{x})} \1 \bigl\{ \mathcal{I}(\vect{x}) =
\mathcal{I} \bigr\} \,\dd\vect{x}
\\
& =& \frac{1}{\prod_{k=1}^N  (N_k! k!^{N_k} )} \int_{\Lambda^N} \e^{-\beta U_N(\vect{x})} \1 \bigl
\{\mathcal{I}(\vect{x}) = \mathcal {I}^\ssup{0} \bigr\} \,\dd\vect{x}.
\end{eqnarray*}
The indicator function in the last integral can be upper bounded by
dropping the requirement that clusters have mutual distance $\geq R$.
This leads to a product of indicator functions, one for each cluster,
encoding that the cluster is connected and stays inside $\Lambda$.
Noting that
\[
\frac{1}{k!} \int_{\Lambda^N} \e^{-\beta U(x_1,\ldots,x_k)} \1 \bigl\{\{
x_1,\ldots,x_k\} \mbox{ connected} \bigr\} \,\dd
x_1\cdots\dd x_k \leq |\Lambda| Z_k^\cl(
\beta)
\]
(integrate first over $x_2,\ldots,x_k$ at fixed $x_1$, and then over
$x_1$), we deduce equation~\eqref{eq:constrained-pf-bound}.

Next, we note that $n! \geq(n/\e)^n$ for all $n\in\N$. Therefore,
\eqref{eq:constrained-pf-bound} gives that
%
%
\begin{equation}
\label{eq:constr-ub2} \qquad Z_\Lambda(\beta,N,N_1,\ldots,N_N)
\leq\exp \biggl( - \beta|\Lambda| f^\ideal \biggl( \beta,
\frac{N}{|\Lambda|}, \biggl( \frac
{N_k}{|\Lambda|} \biggr)_{k\in\N} \biggr)
\biggr),
\end{equation}
where we have set $N_k=0$ for $k \geq N+1$, and $f^\ideal$ is defined
in \eqref{eq:fideal-def}.

Now we turn to a lower bound for the rate function. Let $\mathcal{O}
\subset M_\rho$ be an open set. For $N\in\N$, let $\vectt{\rho}^\ssup
{N}$ be a cluster size distribution in $M_\rho$ of the form $\rho
_k^\ssup{N} = N_k / |\Lambda_N|$ with integer $N_k$, and minimizing
$f^\ideal(\beta,N/|\Lambda|, \vectt{\rho})$ among distributions of this
type. Summing equation~\eqref{eq:constr-ub2} over partitions related to
$\mathcal{O}$, we obtain
\begin{eqnarray*}
- \inf_{\mathcal{O}} I_{\beta,\rho} &\leq&\liminf
_{N\to\infty}\frac
{1}{|\Lambda_N|}\log\P_{\beta,\Lambda_N}^\ssup{N}(
\vectt{\rho }_{\Lambda_N} \in\mathcal{O})\\
& \leq& - \beta\liminf
_{N\to\infty} f^\ideal \biggl(\beta,\frac{N}{|\Lambda_N|},
\vectt{\rho}^\ssup{N} \biggr) + \beta f(\beta,\rho).
\end{eqnarray*}
We have used that the number of integer partitions of $N$, by the
Hardy--Ramanujan formula, is of order $\exp(O(\sqrt{N}))$ and therefore
does not contribute at the exponential scale considered here. Since
$M_\rho$ is compact, we may assume, up to choosing subsequences, that
$\rho_k^\ssup{N} \to\rho_k$ for all $k$, that is, $\vectt{\rho}^\ssup{N}$
converges to some $\vectt{\rho} \in M_\rho$. Since the functional $(\rho, \vectt{\rho})\mapsto f^\ideal(\beta,\rho,\vectt{\rho})$ is lower
semi-continuous, it follows that, along the chosen subsequence,
\[
f^\ideal(\beta,\rho,\vectt{\rho}) =\liminf_{N\to\infty}
f^\ideal \biggl(\beta,{\frac{N}{|\Lambda_N|}}, \vectt{\rho}^\ssup{N}
\biggr) \geq\inf_{\overline{\mathcal{O}}} f^\ideal(\beta,\rho,\cdot).
\]
We deduce
\[
\inf_{\mathcal{O}} f(\beta,\rho,\cdot) \geq\inf_{\overline{\mathcal
{O}}}
f^\ideal(\beta,\rho,\cdot),
\]
for every open set $\mathcal{O}\subset M_\rho$. To conclude, for $\vectt
{\rho} \in M_\rho$, noting that $M_\rho$ is metrizable, we can choose
open environments $\mathcal{O} \searrow\{\vectt{\rho}\}$ and complete
the proof by exploiting the lower semi-continuity of $f^\ideal(\beta,\rho,\cdot)$.
\end{pf}

Our second bound controls the error when dropping the excluded-volume
effect. This was much easier in \cite{CKMS10} and was hidden in the
proof of Proposition 2.2 there.

%
\begin{prop}[(Upper bound)] \label{prop:ub} For each $k \in\N$, let
$a_k>0$ be
such that $(a_k+R)^d < k/ \rho$. Then, for any $\vectt{\rho}=(\rho
_k)_{k\in\N}$,
%
%
\begin{eqnarray}
\label{ffaesti} %
f (\beta,\rho, \vectt{\rho} ) & \leq&
\sum_{k\in\N} k \rho_k f_k^{\cl,a_k}(
\beta) + \biggl(\rho- \sum_{k \in\N} k \rho_k
\biggr) f_\infty^\cl(\beta,\rho) + \frac{1}\beta\sum
_{k\in\N}\rho_k \log\rho
\nonumber
\\[-8pt]
\\[-8pt]
\nonumber
&&{} + \frac{1}\beta\sum_{k \in\N}
\rho_k \biggl( - \log\biggl(1- \frac\rho k (a_k+R)^d
\biggr) +\log\biggl(1+\frac{R}{a_k}\biggr)^d \biggr).
\end{eqnarray}
\end{prop}

\begin{pf}
We first remark that it is enough to show \eqref{ffaesti} for $\vectt
{\rho}$ replaced by $\frac\rho k \mathbf{e}^\ssup{k}$ for any
$k\in\N$ [where $\mathbf{e}^{\ssup k} = (\delta_{k,j})_{j \in\N}$]
and for $\vectt{\rho}$ replaced by $\mathbf{0}$, the sequence consisting
of zeros. Indeed, recall from Theorem~\ref{thm:ldp} that $f(\beta,\rho,\cdot)$ is convex, note that an arbitrary $\vectt{\rho}$ can be written
as the convex combination
\[
(\rho_k)_{k\in\N} = \sum_{k \in\N}
\frac{k \rho_k}{\rho} {\frac
{\rho} {k}}\mathbf{e}^\ssup{k} + \biggl( 1 -
\sum_{k\in\N} \frac
{k\rho_k}{\rho} \biggr) \mathbf{0},
\]
and note that the right-hand side of \eqref{ffaesti} is affine in $\vectt
{\rho}$. Hence, we only have to show that
%
%
\begin{eqnarray}
\label{eq:one-finite-size} &&f \biggl(\beta,\rho, {\frac\rho k \mathbf{e}^\ssup{k}}
\biggr)\nonumber\\
&&\qquad \leq\rho f_k^{\cl,a_k}(\beta) + \frac{1}\beta
\frac{\rho}{k} \log\rho + \frac{1}\beta\frac{\rho}{k} \biggl( -
\log \biggl(1- \frac\rho k (a_k+R)^d \biggr)
\\
\nonumber&&\qquad\quad{}+\log \biggl(
\biggl(1+\frac{R}{a_k}\biggr)^d \biggr) \biggr),\qquad
k\in \N,
\end{eqnarray}
and that
%
%
\begin{equation}
\label{eq:only-large} f (\beta,\rho, \mathbf{0} ) \leq\rho f_\infty^\cl(
\beta,\rho).
\end{equation}
We now prove \eqref{eq:only-large}. Let $\mathcal{O} \subset M_\rho$ be
an open set containing
$\mathbf{0}$, and $\overline{\mathcal{O}}$ its closure. By the LDP,
\[
\limsup_{N\to\infty} \frac{1}{|\Lambda_N|} \log\P_{\beta,\Lambda
_N}^\ssup{N}(
\bolds{\rho}_{\Lambda_N} \in\overline{\mathcal{O}}) \leq- \inf
_{\overline{\mathcal{O}}} I_{\beta,\rho}.
\]
For $N \in\N$, consider the cluster size distribution obtained by
putting all particles
into one large cluster, $\rho_1^\ssup{N} = \cdots= \rho_{N-1}^\ssup{N}=0$,
$\rho_N^\ssup{N} = 1/|\Lambda_N|$. Note that $\bolds{\rho}^\ssup
{N}=(\rho_k^\ssup{N})_{k\in\N}$ lies in $M_\rho$ for any $N\in\N$. We
have $\bolds{\rho}^\ssup{N} \to0$ as
$N \to\infty$ and thus $\bolds{\rho}^\ssup{N} \in\mathcal{O}
\subset\overline{\mathcal{O}}$ for sufficiently large $N$. As a
consequence, we can lower bound
\[
\P_{\beta,\Lambda_N}^\ssup{N}( \bolds{\rho}_{\Lambda_N} \in
\overline{\mathcal{O}}) \geq\P_{\beta,\Lambda_N}^\ssup{N}\bigl( \bolds{
\rho}_{\Lambda_N} = \bolds{\rho}^\ssup{N}\bigr) = \frac{|\Lambda_N| Z_N^{\cl,L_N}(\beta)}{Z_{\Lambda_N}(\beta,N)}.
\]
Recalling that $|\Lambda_N|=N/\rho$, it follows that
\[
- \inf_{\overline{\mathcal{O}}} I_{\beta,\rho} \geq\limsup
_{N\to\infty
} \frac{1}{|\Lambda_N|} \log\frac{|\Lambda_N| Z_N^{\cl,L_N}(\beta
)}{Z_{\Lambda_N}(\beta,N)} \geq- \beta
\rho f_\infty^\cl(\beta,\rho) + \beta f(\beta,\rho).
\]
Since $I_{\beta,\rho}(\cdot)=\beta f(\beta,\rho,\cdot)-\beta f(\beta,\rho)$, this implies $\inf_{\overline{\mathcal{O}}} f(\beta,\rho,\cdot
) \leq\rho f_\infty^\cl(\beta,\rho)$.
This holds for all open sets $\mathcal{O}$ containing $\mathbf{0}$.
Letting $\mathcal{O} \searrow\{\mathbf{0}\}$ and using the lower
semi-continuity of $f(\beta,\rho,\cdot)$, we deduce \eqref{eq:only-large}.

Now let us turn to \eqref{eq:one-finite-size}. We proceed in a way
that is analogous to Lemma~\ref{lem:nonempty-int}. Fix $k \in\N$. Let
$N$ be a multiple of $k$.
Consider the cluster size distribution obtained by putting all
particles into clusters of size
$k$, that is, put $N_j^\ssup{N}= (N/k) \delta_{j,k}$ for $j\in\N$. We
divide the box $\Lambda_N$ into
$\ell_N$ boxes of side length $a_k$ with mutual distance at least $R$.
Hence, $\ell_N\sim\frac{N}\rho(a_k+R)^{-d}$. The assumption $(a_k+R)^d
<k/\rho$ guarantees that $\ell_N>N/k$ for sufficiently large $N$.
Therefore, we can lower bound
\[
Z_{\Lambda_N}\bigl(\beta,N,N_1^\ssup{N},
\ldots,N_N^\ssup{N}\bigr) \geq\pmatrix{\ell_N
\cr
N/k} \bigl( a_k^d Z_k^{\cl,a_k}(\beta)
\bigr)^{N/k}.
\]
Therefore, using that $|\Lambda_N|=N/\rho$ and Stirling's formula,
%
%
\begin{eqnarray}
\label{Zbound} %
&&\liminf_{N\to\infty}
\frac{1}{|\Lambda_N|} \log Z_{\Lambda_N}\bigl(\beta,N,N_1^\ssup{N},
\ldots,N_N^\ssup{N}\bigr)
\nonumber
\\[-8pt]
\\[-8pt]
\nonumber
&&\qquad
\geq\frac{\rho}{k} \log Z_k^{\cl,a_k}(\beta) -
\frac{\rho}{k} \log \rho + \frac{\rho}{k} \log \biggl( \frac{a_k^d}{(a_k+R)^d}
- \frac{\rho
a_k^d}{k} \biggr).
\end{eqnarray}
Multiplying the right-hand side with $- \beta^{-1}$, the right-hand
side of \eqref{eq:one-finite-size} arises. In the same way as in the
proof of \eqref{eq:only-large}, one derives, with the help of Lemma~\ref
{lem-Delta}, that $f(\beta,\rho,\frac\rho k \vect{e}^{\ssup k})$ is not
larger than $- \beta^{-1}$ times the left-hand side of \eqref{Zbound}.
This ends the proof of \eqref{eq:one-finite-size}.
\end{pf}

\section{Bounds for the cluster free energy}\label{sec-finftyfinite}

In this section we give some more bounds that will later be
used in the proofs of Theorems~\ref{thm:main} and \ref{thm:unif}. We
further estimate some entropy terms, and we give bounds that control
the replacement of temperature-depending terms by the corresponding
ground-state terms. Throughout this section we assume that the pair
potential $v$ satisfies Assumption \ref{assV}.

We will later replace the term $\sum_k \rho_k (\log\rho_k - 1)$ in
$f^\ideal(\beta,\rho, (\rho_k)_k)$ by $\sum_k \rho_k \log\rho_k$. To
this aim the following will be useful.

%
\begin{lemma}[(Entropy bound)] \label{lem:entropy}
For any probability distribution $(p_k)_{k\in\N}$ on $\N$,
\[
0 \leq-\sum_{k\in\N} p_k \log
p_k \leq1+ \log\sum_{k\in\N} k
p_k.
\]
\end{lemma}

\begin{pf} We may assume that the expectation $\sum_{k\in\N}kp_k$ is finite.
It is elementary to see that the maximizer of the entropy among the
set of
probability distributions with a given finite expectation is a
geometric distribution.
For $p_k = (1-u) u^{k-1}$, $k \in\N$, for some $u\in(0,1)$, the
expectation is
$\sum_{k\in\N} k p_k= 1/(1-u)$, and the entropy is
\begin{eqnarray*}
-\sum_{k\in\N}p_k \log p_k & =&
- \log(1-u) - (1-u) \sum_{k\in\N} u^{k-1} (k-1)
\log u
\\
& = &- \log(1-u) - \frac{u \log u}{1-u} = \log\sum_{k\in\N}kp_k
+ \frac
{u \log u}{u-1}. 
\end{eqnarray*}
We conclude by observing that $x \log x \geq x-1$ for all $x >0$ and
recalling that $u<1$.
\end{pf}

%
\begin{lemma} \label{lem:entropy2}
For any $\rho\in(0,\infty)$ and any $\vectt{\rho}=(\rho_k)_{k\in\N}\in
M_\rho$,
\[
\sum_{k\in\N} \rho_k \log\frac{\rho_k}{\rho}
\geq- 2 \rho.
\]
\end{lemma}

\begin{pf} Put $m:= \sum_{k\in\N} \rho_k$ and $p_k:= \rho_k/ m$. Then
\begin{eqnarray*}
 \sum_{k\in\N}
\rho_k \log\frac{\rho_k}{\rho} &= &\sum_{k\in\N}
m p_k \log\frac{m p_k}{\rho} = m \log\frac{m}{\rho} + m \sum
_{k\in\N} p_k \log p_k
\\
&\geq& m \log\frac{m}{\rho}-m-m\log\sum_{k\in\N} k
p_k\geq2m\log\frac{m}\rho-m,
\end{eqnarray*}
where we applied Lemma~\ref{lem:entropy} and that $\sum_{k\in\N} k p_k
\leq\rho/ m$. Now use the inequality $x \log x \geq x-1$ and drop the
term $m$.
\end{pf}

In our bounds in Lemma~\ref{prop:lb} and Proposition~\ref{prop:ub}, we
will later replace the cluster-free energies with ground state
energies; in this section we give bounds that will allow us to control
the replacement error. We also prove that $f_\infty^\cl(\beta)$ and
$f_\infty^\cl(\beta,\rho)$ are finite.

%
\begin{lemma}[{[Lower bound for $f_k^\cl(\beta)$ and $f_\infty^\cl(\beta
)$]}] \label{lem:cfe-lower-bound}
There is a constant $C>0$ such that for all $\beta\in(0,\infty)$,
\[
f_k^\cl(\beta) \geq\frac{E_k}{k} -
\frac{C} \beta,\qquad k \in\N,\beta \in(0,\infty).
\]
In particular, $f_\infty^\cl(\beta) \geq e_\infty- \frac{C} \beta$ for
any $\beta\in(0,\infty)$.
\end{lemma}

\begin{pf} We follow \cite{CKMS10}, Section~2.4. First, note that
\[
Z^\cl_k (\beta) \leq\e^{-\beta E_k} \frac{1}{k!}
\thinspace\bigl | \bigl\lbrace(x_2,\ldots,x_k)\in\bigl(
\R^d\bigr)^{k-1} \dvtx\{ 0,x_2,
\ldots,x_k\}\ R\mbox{-connected} \bigr\rbrace\bigr |
\]
with $|\cdot|$ the Lebesgue volume. Now, with each $\vect{x}' =
(x_2,\ldots,x_k)$
such that $\vect{x}:=(0,\vect{x}')$ is $R$-connected, we can associate
a tree $T(\vect{x}')$
with vertex set $\{1,\ldots,k\}$ and edge set $E(T(\vect{x}'))\subset
\{ \{i,j\} \dvtx i \neq j\}$,
and such that
\[
\{i,j\} \in E\bigl(T\bigl(\vect{x}'\bigr)\bigr)
\quad\Longrightarrow\quad|x_i - x_j| \leq R.
\]
Note that for a given $\vect{x}'$, there are in general several trees
satisfying this
condition; we pick arbitrarily one of them and call it $T(\vect{x}')$.
Now we have
\begin{eqnarray*}
&&\hspace*{-4pt} \bigl| \bigl\lbrace\vect{x}' \in\bigl(\R^d
\bigr)^{k-1} \mid\bigl(0,\vect{x}'\bigr)\ R\mbox
{-connected} \bigr\rbrace \bigr|
\\
&&\hspace*{-6pt}\qquad = \sum_{T\ \mathrm{tree}} \bigl| \bigl\lbrace\vect{x}'
\in\bigl(\R^d\bigr)^{k-1} \mid\bigl(0,\vect{x}'
\bigr)\ R\mbox {-connected}, T\bigl(\vect{x}'\bigr)= T \bigr\rbrace\bigr |
\\
&&\hspace*{-6pt}\qquad \leq\sum_{T\ \mathrm{tree}} \bigl| \bigl\lbrace
\vect{x}' \in\bigl(\R^d\bigr)^{k-1} \mid
\bigl(0,\vect{x}'\bigr)\ R\mbox{-connected}, \{i,j\} \in E(T)
\Rightarrow|x_j-x_i| \leq R \bigr\rbrace\bigr |.
\end{eqnarray*}
For each given tree $T$, the Lebesgue volume of the set in the last
line above is
upper bounded by $|B(0,R)|^{k-1}$. By Cayley's theorem (see~\cite{AZ98},
pages 141--146), the number of labeled trees with $k$ vertices
is $k^{k-2}$. Thus
\[
Z_k^\cl(\beta) \leq\e^{-\beta E_k} \frac{k^{k-2}}{k!}
\bigl|B(0,R)\bigr|^{k-1},
\]
and the proof is easily completed.
\end{pf}

Now we show that the volume constraint in the cluster partition
function is immaterial for large $\beta$ if the radius of the confining
box is of order of the particle number with a sufficiently large prefactor.

%
\begin{lemma} [{[Low-temperature behavior of $f_k^{\cl,a}(\beta)$]}] \label
{lem:fk-low-temp}
For any $k \in\N$ and any choice of
$a_k(\beta)$ in $[k R,\infty)$,
\[
\lim_{\beta\to\infty} f_k^{\cl,a_k(\beta)} (\beta) =
\frac{E_k}{k}.
\]
\end{lemma}

\begin{pf} The lower bound ``$\geq$''~is trivial since $Z_k^{\cl,a}(\beta) \leq Z_k^\cl(\beta)$ for any $a$. For $a_k(\beta) \geq k R$,
the box $[0,a_k(\beta)]^d$ is certainly large enough to contain a
minimiser of $\vect{x}\mapsto U_k(\vect{x})$. Therefore, lower bounding
the integral by an integral in a neighborhood of the minimiser, we find
\[
\liminf_{\beta\to\infty} \frac{1}{\beta} \log Z_k^{\cl,a_k(\beta)}
\geq- \frac{E_k}k,
\]
which is the upper bound ``$\leq$''.
\end{pf}

Under additional assumptions, most importantly Assumption~\ref
{ass:maxdist}, it will be enough to pick $a_k$ of order $k^{1/d}$
instead of $k$, with some error of order $\frac{1}\beta\log\beta$:

%
\begin{lemma} [{[Uniform low-temperature bounds for $f_k^{\cl,a}(\beta)$]}]
\label{lem:akunif}
Suppose that the pair potential also satisfies Assumptions~\ref{ass:hoelder}
and~\ref{ass:maxdist}. There is an $\alpha>0$ and a $\overline{\beta} >0$
such that for all $\beta\in[\overline\beta,\infty)$, and
every sequence of $a_k$'s satisfying $a_k > \alpha k^{1/d}$,
%
\begin{equation}
\label{eq:akunif-a} f_k^{\cl,a_k}(\beta) \leq\frac{E_k}k +
\frac{C}\beta\log\beta, \qquad k\in\N.
\end{equation}
In particular, for any $\rho\in(0,1/\alpha^d)$ and $\beta\in
[\overline\beta,\infty)$,
%
%
\begin{equation}
\label{eq:akunif-b} f_\infty^{\cl}(\beta,\rho) \leq
e_\infty+ \frac{C}\beta\log\beta.
\end{equation}
\end{lemma}

\begin{pf} The strategy of the proof is as follows. According to
Assumption~\ref{ass:maxdist}, we may pick a minimiser for $U_k$ that
fits into some ball whose volume is of order of the particle number.
Then we restrict the integral in the definition of the cluster
partition function to some neighbourhood of this minimiser and control
the error with the help of the H\"older continuity from Assumption~\ref
{ass:hoelder}. Let us turn to the details.

Let $c>0$ be as in Assumption~\ref{ass:maxdist}, $\delta>0$ as in
Lemma~\ref{lem:zka-lb}.
Then $\alpha:=2(c+ \delta)$ satisfies $\alpha k ^{1/d} \geq\delta+
ck^{1/d}$ for all $k \in\N$.
Fix $t\in(1,R/b)$. Let $n_\mathrm{max}\in\N$ be the maximal number of
particles that can
be placed in $B(0,R)$, keeping
mutual distance $\geq r_\mathrm{min}$, with $r_\mathrm{min}$ as in
Assumption~\ref{ass:hoelder}.

For $k \in\N$, let $a_k > \alpha k^{1/d}$ and let $\vect{x}^{\ssup
0}= (x_1^{\ssup0},\ldots,x_k^{\ssup0})$ be a minimiser of the energy $U_k$
that fits into the cube with side length $a_k - \delta$. Thus $\vect
{x}^{\ssup0}$ is $b$-connected, and $|x_i - x_j| \geq r_\mathrm{min}$
for every $i \neq j$. The scaled state
$ t \vect{x}^{\ssup0}$ is $t b$-connected and has minimum
interparticle distance
$ \geq t r_\mathrm{min}$. By the H{\"o}lder continuity of the
potential~$v$,
\begin{eqnarray*}
&&\bigl|U\bigl(t \vect{x}^{\ssup0}\bigr) - U\bigl(\vect{x}^{\ssup0}\bigr)\bigr|
\\
&&\qquad \leq\frac{1}{2} \sum_{i=1}^k
\sum_{j \neq i} \bigl| v\bigl(t\bigl|x_i^{\ssup0}
- x_j^{\ssup0}\bigr|\bigr) - v\bigl(\bigl|x_i^{\ssup0}
- x_j^{\ssup0}\bigr|\bigr)\bigr |
\\
&&\qquad \leq k n_\mathrm{max} \sup \bigl\{ \bigl|v\bigl(r'\bigr) -
v(r)\bigr|\dvtx r \geq r_\mathrm{min}, r' \geq
r_\mathrm{min},\bigl |r - r'\bigr|\leq(t-1)b \bigr\}
\\
&&\qquad \leq C k n_\mathrm{max} (t-1)^s b^s
\end{eqnarray*}
with $C$ and $s$ such that $|v(r') - v(r)|\leq C |r'-r|^s$ for any
$r,r'\geq r_\mathrm{min}$.
Let $\eps\in(0,1)$ such that
\[
\eps\leq\delta/2,\qquad r_\mathrm{min} \leq t r_\mathrm{min} - 2 \eps\quad
\mbox{and}\quad t b + 2 \eps\leq R.
\]
We will obtain a lower bound for $Z_k^{\cl,a_k}(\beta)$ by considering
configurations $(x_1,\ldots,x_k)$ with exactly one particle per $\eps
$-ball around $t x_j^{\ssup0}$ for $j=2,\ldots,k$. To this end, put
\[
\mathcal{M}':= \bigcup_{\sigma\in\Sym'_{k-1}} \bigl( B
\bigl( t x_{\sigma(2)}^{\ssup0},\eps\bigr) \times\cdots\times B\bigl( t
x_{\sigma(k)}^{\ssup0},\eps\bigr) \bigr),
\]
where $\Sym'_{k-1}$ denotes the set of permutations of $2,\ldots,k$, and
let $\mathcal{M}$
be the set of configurations in the cube of side length $a_k - \delta$
obtained by rigid shifts from configurations in $\{x_1^\ssup{0}\}
\times\mathcal{M}'$.
For small enough $\eps$, the balls $B(t x_{\sigma(2)}^{\ssup0},\eps
),\ldots,B(t x_{\sigma(k)}^{\ssup0},\eps)$ do not overlap, and $\mathcal
{M}'$ has therefore Lebesgue volume $(k-1)! |B(0,\eps)|^{k-1}$. Moreover,
\[
|\mathcal{M}| \geq\bigl|\mathcal{M}'\bigr| \bigl(a_k - \delta- c
k^{1/d}\bigr)^d \geq \frac{a_k^d}{2} \bigl|
\mathcal{M}'\bigr|.
\]
Now $\vect{x} \in\mathcal{M}$ is $R$-connected
and has minimum interparticle distance $\geq r_\mathrm{min}$.
Thus
\[
\bigl|U(\vect{x}) - U\bigl(t\vect{x}^{\ssup0}\bigr)\bigr| \leq C k
n_\mathrm{max} \eps ^s,\qquad \vect{x}\in\Mcal.
\]
Restricting the integral in the definition \eqref{Zcldef} of $Z_k^{\cl,a_k}(\beta)$ to $\Mcal$, we obtain
\[
a_k^d Z_k^{\cl,a_k}(\beta) \geq
\frac{a_k^d}{2 k} \bigl|B(0,\eps)\bigr|^{k-1} \exp ( - \beta\bigl(E_k
+ C k n_\mathrm{max} \bigl[\eps ^s+ (t-1)^s
b^s\bigr] \bigr).
\]
This implies, for $|B(0,\eps)|\leq1$,
\[
f_k^{\cl,a_k}(\beta) \leq\frac{E_k}{k}+\frac{C n_\mathrm{max} (\eps^s
+(t-1)^s b^s)}{\beta}-
\frac{1}{\beta} \log\bigl|B(0,\eps)\bigr| + \frac{\log
2}{\beta}.
\]
Now we pick $\eps=1/\beta$ for definiteness and obtain that \eqref
{eq:akunif-a} is satisfied for sufficiently large $\beta$.
\end{pf}

%
%

\section{Proof of \texorpdfstring{$\Gamma$}{$Gamma$}-convergence and uniform bounds}\label{sec-lastproofs}

In this section, we prove Theorems~\ref{thm:main} and \ref
{thm:unif}. Recall that Theorem~\ref{thm:main} is proved under the sole
Assumption \ref{assV} and that we additionally suppose that Assumptions~\ref
{ass:hoelder} and \ref{ass:maxdist} hold for Theorem~\ref{thm:unif}.

\subsection{Proof of Theorem~\texorpdfstring{\protect\ref{thm:main}}{1.2}}\label{sec-proofThm14}

Fix $\nu\in(0,\infty)$, and let $(0,\infty) \ni s \mapsto(\beta(s),\rho
(s))$ be a curve in $(0,\infty)^2$ such that, as $s\to\infty$,
\[
\beta(s)\to\infty, \rho(s) \to0,\qquad -\frac{1}{\beta(s)} \log \rho(s) \to\nu.
\]
We need to show that, for any $\vect{q} = (q_k)_{k\in\N} \in\mathcal{Q}$,
\begin{enumerate}
\item[] \emph{Lower bound}: For all curves $\vect{q}^\ssup{s} \to\vect{q}$,
%
%
\begin{equation}
\label{eq:gamma-lb} \liminf_{s\to\infty} \frac{1}{\rho(s)} f\bigl(
\beta(s),\rho(s), \vectt{\rho}^\ssup{s}\bigr) \geq g_\nu (
\vect{q}).
\end{equation}
\item[] \emph{Upper bound/recovery sequence}: there is a curve $\vect
{q}^\ssup{s} \to\vect{q}$ such that
%
%
\begin{equation}
\label{eq:gamma-ub} \limsup_{s\to\infty} \frac{1}{\rho(s)}f\bigl(
\beta(s),\rho(s), \vectt{\rho }^\ssup{s}\bigr) \leq g_\nu(
\vect{q}).
\end{equation}
\end{enumerate}

\begin{pf*}{Proof of the lower bound}
We write $\vect{q}^\ssup{s} =
(q_k^\ssup{s})_k \in\mathcal{Q}$.
Define $\vectt{\rho}^\ssup{s} = (\rho_k^\ssup{s})_{k\in\N}$ by
$q_k^\ssup{s}= k\rho_k^\ssup{s}/\rho$. Let $C>0$ and $\overline\beta>0$
such that
$k f_k^\cl(\beta) \geq E_k - C k\beta^{-1}$ for any
$k\in\N\cup\{\infty\}$ and $\beta\in[\overline\beta,\infty)$; see
Lemma~\ref{lem:cfe-lower-bound}.
Then Lemma~\ref{prop:lb} gives
\begin{eqnarray*}
\frac{1}{\rho(s)}f\bigl(\beta(s),\rho(s), \vectt{\rho}^\ssup{s}\bigr)
& \geq&\sum_{k\in\N} \frac{\rho_k^\ssup{s}}{\rho(s)} E_k +
\biggl(1- \sum_{k\in\N} k \frac{\rho^\ssup{s}_k}{\rho(s)} \biggr)
e_\infty\\
&&{}+ \frac{1}{\beta(s)} \sum_{k\in\N}
\frac{\rho^\ssup{s}_k}{\rho
(s)} \bigl(\log\rho^\ssup{s}_k - 1 \bigr) -
\frac{C}{\beta(s)}
\\
& =& \sum_{k\in\N} \frac{\rho^\ssup{s}_k}{\rho(s)}
\biggl(E_k-\frac{1}{\beta(s)} \log\rho(s) \biggr) + \biggl(1- \sum
_{k\in\N} k \frac{\rho^\ssup{s}_k}{\rho(s)} \biggr) e_\infty
\\
&&{} + \frac{1}{\beta(s)} \sum_{k\in\N} \frac{\rho^\ssup{s}_k}{\rho
(s)}
\biggl(\log\frac{\rho^\ssup{s}_k}{\rho(s)} - 1 \biggr) - \frac{C}{\beta(s)}.
\end{eqnarray*}
The term in the second line converges to $g_\nu(\vect{q})$ because of
the continuity of the map $\vect{q} \mapsto\sum_{k\in\N} q_k (E_k -
\nu)/k + (1- \sum_{k\in\N} q_k) e_\infty$; here enters the property
$E_k/k \to e_\infty$. The terms in the last line are, by Lemma~\ref
{lem:entropy2}, of order $1/\beta(s)$
and therefore converge to $0$.
\end{pf*}

\begin{pf*}{Proof of upper bound/existence of a recovery sequence}
We choose $\rho$-dependent box sizes $a_k(\rho)$ such that
$(a_k (\rho) + R)^d <k / (2\rho)$, $a_k >R$, and $a_k > \delta+ k
^{1/d}(r_\mathrm{hc} +\delta)$,
with $\delta$ as in Lemma~\ref{lem:zka-lb}. Such a choice is possible
for small enough $\rho$,
and compatible with the additional requirement that $a_k(\rho) \to
\infty$ as $\rho\to0$,
for every $k\in\N$. Lemma~\ref{lem:zka-lb} tells us that
\[
f_k^{\cl,a_k(\rho(s))} \leq C(\delta) - \frac{1}{\beta(s)} \log \bigl|B(0,
\delta/2)\bigr| + \frac{\log(k/\rho(s)) }{d \beta(s) k},
\]
which can be upper bounded by some constant $C$, uniformly in $k \in\N
$ and sufficiently large $s$.

Now we apply Proposition~\ref{prop:ub}. This gives,
for sufficiently large $s$ and any sequence $\vectt{\rho}=(\rho_k)_k$,
%
%
\begin{eqnarray}
&&\frac{1}{\rho(s)}f\bigl(\beta(s),\rho(s), \vectt{\rho}\bigr)\nonumber \\
&&\qquad \leq\sum
_{k\in\N} k \frac{\rho_k}{\rho(s)} f_k^{\cl,a_k(\rho(s))}
\bigl(\beta(s)\bigr) + \biggl(1- \sum_{k\in\N} k
\frac{\rho_k}{\rho(s)} \biggr) f_\infty^\cl \bigl(\beta(s),\rho(s)
\bigr)
\\
& &\qquad\quad{}+ \frac{1}{\beta(s)} \sum_{k\in\N} \frac{\rho_k}{\rho(s)}
\log \rho(s) + \frac{1}{\beta(s)} \sum_{k\in\N}
\frac{\rho_k}{\rho(s)} (d+1) \log2.\nonumber
\end{eqnarray}
Consider first the case $\sum_{k=1}^\infty q_k =1$.
Let $\vect{q}^\ssup{s}:=\vect{q}$. We have, for any $K\in\N$,
\begin{eqnarray*}
&&\frac{1}{\rho(s)}f\bigl(\beta(s),\rho(s), \vectt{\rho}^\ssup{s}\bigr)
\\
&&\qquad\leq\sum_{k=1}^K q_k
\biggl(f_k^{\cl,a_k(\rho(s))}\bigl(\beta(s)\bigr) - \frac
{\log\rho(s)}{\beta(s) k}
\biggr) \\
&&\qquad\quad{}+ C \sum_{k=K+1}^\infty q_k
+ \frac{\log
2^{d+1}}{\beta(s)}.
\end{eqnarray*}
Since $a_k(\rho(s)) \to\infty$ as $s\to\infty$ for any $k\in\{1,\ldots,K\}$, using Lemma~\ref{lem:fk-low-temp}, we get
\[
\limsup_{s\to\infty} \frac{1}{\rho(s)} f\bigl(\beta(s),\rho(s),
\vectt{\rho }^\ssup{s}\bigr) \leq\sum_{k=1}^K
q_k \frac{E_k- \nu}{k} + C \sum_{k=K+1}^\infty
q_k.
\]
Letting $K\to\infty$ we find that $\limsup_{s\to\infty} \rho(s)^{-1}
f(\beta(s),\rho(s), \vectt{\rho}^\ssup{s}) \leq g_\nu(\vect{q})$.

Next, consider the case $q_k=0$ for all $k \in\N$.
For $n\in\N$, let $s_n>0$ large enough so that for $s\geq s_n$,
$|f_n^{\cl,a_n(\rho(s))} - E_n/n| \leq1/n$. The sequence $(s_n)_{n\in
\N}$ can be chosen increasing and diverging. We set $k(s):=n$ for $s\in
[s_n,s_{n+1})$ and $n\in\N$. It follows that $k(s) \to\infty$ as $s
\to\infty$, and
\[
\biggl|f_{k(s)}^{\cl,a_{k(s)}(\rho(s))}\bigl(\beta(s)\bigr) - \frac
{E_{k(s)}}{k(s)} \biggr|
\leq\frac{1}{k(s)},\qquad s\in[s_1,\infty),
\]
from which we deduce $f_{k(s)}^{\cl,a_{k(s)}(\rho(s))}(\beta(s)) \to
e_\infty$ as $s\to\infty$.
Set $q_k(s):= \delta_{k,k(s)}$. Then we find
\[
\limsup_{s\to\infty} \frac{1}{\rho(s)} f\bigl(\beta(s),\rho(s),
\vectt{\rho }^{\ssup s}\bigr) \leq e_\infty= g_\nu(
\vect{q}).
\]
To conclude, we observe that every $\vect{q}\in\mathcal{Q}$ can be
written as a convex
combination of a vector $\vect{q}'$ with $\sum_{k\in\N} q'_k=1$ and
the zero vector, and a recovery sequence is
constructed by taking the convex combination of $\vect{q}'$ and the
recovery sequence for the zero vector.
\end{pf*}

\subsection{Proof of Theorem \texorpdfstring{\protect\ref{thm:unif}}{1.8}}\label
{sec-proofThm17}
\mbox{}

\textit{Proof of} (1):
We prove \eqref{estiratefct} in terms of $\rho_k$'s instead of $q_k$'s.
Then it reads
%
%
\begin{eqnarray}
\label{estiratefctrho} &&\biggl| f\bigl(\beta,\rho, (\rho_k)_{k\in\N}\bigr) -
\biggl[\sum_{k\in\N} \rho_k
\biggl(E_k + \frac{\log\rho}{\beta} \biggr) + \biggl( \rho- \sum
_{k\in\N} k \rho_k \biggr) e_\infty \biggr] \biggr|
\nonumber
\\[-8pt]
\\[-8pt]
\nonumber
&&\qquad\leq\frac{C}{\beta} \rho\log\beta,\qquad (\rho_k)_{k\in\N}\in
M_\rho.
\end{eqnarray}
Lemmas~\ref{prop:lb}, \ref{lem:entropy2}, and~\ref{lem:cfe-lower-bound}
yield that there is $C\in(0,\infty)$ such that, for all $\beta,\rho\in
(0,\infty)$ and $\vectt{\rho}=(\rho_k)_{k\in\N}\in M_\rho$,
%
%
\begin{eqnarray}
\label{eq:unif-lb} %
f(\beta,\rho, \vectt{\rho}) &\geq&
f^\ideal(\beta,\rho, \vectt{\rho})\nonumber
\\
&\geq&\sum_{k\in\N} k\rho_k \biggl(
\frac{E_k}k-\frac{C}\beta \biggr) + \biggl(\rho -\sum
_{k\in\N}k\rho_k \biggr) \biggl(e_\infty-
\frac{C}\beta \biggr)
\nonumber
\\[-8pt]
\\[-8pt]
\nonumber
&&{}+\frac{1}\beta\sum
_{k\in\N}\rho_k\log\frac{\rho_k}\rho+
\frac{\log\rho
-1}\beta\sum_{k\in\N}\rho_k
\\
&\geq&\sum_{k \in\N} \rho_k
\biggl(E_k + \frac{\log\rho}{\beta} \biggr) + \biggl(\rho- \sum
_{k\in\N} k \rho_k \biggr) e_\infty- (C +3)
\frac{\rho
}{\beta}.\nonumber
\end{eqnarray}
This is ``$\geq$'' in \eqref{estiratefctrho}. For proving ``$\leq
$'', we pick, for $\rho\in(0,\infty)$ and $k\in\N$, box diameters
$a_k(\rho)$ such that $a_k (\rho) > \alpha k^{1/d}$, with $\alpha$ as
in Lemma~\ref{lem:akunif}, and $(a_k(\rho) + R)^d < k /2\rho$, for all
$k \in\N$. This is possible provided $\rho<k/2 ( \alpha k ^{1/d} +
R)^d$ for any $k\in\N$, and this is, by monotonicity in $k$, guaranteed
for $\rho<\overline\rho$, where we put $\overline\rho=\frac{1}{2
(\alpha+ R)^d}$. We may also assume, without loss of generality, that
$\alpha>R$, which implies that $a_k(\rho) >R$ for all $k \in\N$. We
obtain, for $(\beta,\rho)\in[\overline{\beta},\infty)\times(0,\overline
\rho)$, and $C>0$ as in Lemma~\ref{lem:akunif}, for any $\vectt{\rho}\in
M_\rho$, with the help of Proposition~\ref{prop:ub},
%
%
\begin{eqnarray}
\label{eq:unif-ub} %
 f(\beta,\rho, \vectt{\rho}) &\leq&\sum
_{k \in\N} k\rho_k \biggl(
\frac
{E_k}k-\frac{C}\beta\log\beta \biggr) + \biggl(\rho- \sum
_{k\in\N} k \rho _k \biggr)
\biggl(e_\infty-\frac{C}\beta\log\beta \biggr)\nonumber\\
& &{}+
\frac{\log\rho}\beta\sum_{k\in\N}\rho_k
+\frac{1}\beta\sum_{k\in\N}\rho_k
\biggl(-\log\biggl(1-{\frac{1}2}\biggr)+\log \frac{k}{2\rho a_k(\rho)^d} \biggr)
\nonumber
\\[-8pt]
\\[-8pt]
\nonumber
&\leq&\sum_{k \in\N} \rho_k
\biggl(E_k + \frac{\log\rho}{\beta} \biggr) + \biggl(\rho- \sum
_{k\in\N} k \rho_k \biggr) e_\infty\\
&&{}+
\frac{C\rho}\beta \log\beta+(d+1)\frac{\rho}\beta\log2, \nonumber
\end{eqnarray}
which is the corresponding upper bound in \eqref{estiratefctrho}.

\textit{Proof of} (2): Let $\vectt{\rho} = (\rho_k)_k$ be a
minimiser of $f(\beta,\rho,\cdot)$ and $\vect{q}:= ( k \rho_k/\rho
)_{k\in\N}$. Write $\nu= - \beta^{-1} \log\rho$. Then
\[
\frac{1}\rho f(\beta,\rho) = \frac{1}\rho f(\beta,\rho, \vectt{
\rho}) \geq g_{\nu} (\vect{q}) - \frac{C} \beta\log\beta\geq\mu(
\nu) - \frac{C}\beta\log\beta.
\]
Similarly, let $\vect{q}$ be a minimiser of $g_\nu(\cdot)$ and $\vectt
{\rho}:= (\rho q_k/k)_{k\in\N}$. Then
\[
\mu(\nu) = g_\nu(\vect{q}) \geq\frac{1}\rho f(\beta,\rho,
\vectt{\rho}) - \frac{C}\beta\log\beta \geq\frac{1}\rho f(\beta,
\rho) - \frac{C}\beta\log\beta.
\]

\textit{Proof of} (3): Let $\vectt{\rho} = (\rho_k)_k$ be a minimiser of
$f(\beta,\rho,\cdot)$
and $\vect{q}:= ( k \rho_k/\rho)_{k\in\N}$. Write $\nu= - \beta^{-1}
\log\rho$. Then (1) and (2) yield
\[
g_{\nu}(\vect{q}) - \mu(\nu) \leq\frac{1}\rho f(\beta,\rho,
\vectt{\rho })+\frac{C} \beta\log\beta- \biggl(\frac{1}\rho f(
\beta,\rho)-\frac{C}\beta\log\beta \biggr)\leq2 \frac{C} \beta
\log\beta.
\]
Hence,
%
%
\begin{eqnarray}
\label{gnuesti2} 2 \frac{C} \beta\log\beta&\geq& g_{\nu}(\vect{q})
- \mu(\nu)
\nonumber
\\[-8pt]
\\[-8pt]
\nonumber
&= &\sum_{k
\in\N} \biggl( \frac{E_k - \nu}{k}
- \mu(\nu) \biggr) q_k + \bigl(e_\infty- \mu(\nu) \bigr)
\biggl(1 - \sum_{k\in\N} q_k \biggr).
\end{eqnarray}
For $\nu< \nu^*$, we use that $\mu(\nu)=e_\infty$ and estimate
\[
\frac{E_k - \nu}{k} - \mu(\nu) = \frac{E_k - k e_\infty- \nu}{k} \geq \frac{\nu^*- \nu}{k}.
\]
Substituting this in \eqref{gnuesti2}, this yields the first claim,
\eqref{eq:smallnu}.

For $\nu>\nu^*$, we restrict the first sum on the right of \eqref
{gnuesti2} to $k\in\N\setminus M(\nu)$, where we lower estimate the
brackets against $\Delta(\nu)$, and we estimate $e_\infty- \mu(\nu)
\geq\Delta(\nu)$. This gives
\[
2 \frac{C} \beta\log\beta\geq\sum_{k\in\N\setminus M(\nu)}
\Delta(\nu ) q_k+\Delta(\nu) \biggl(1 - \sum
_{k\in\N} q_k \biggr)=\Delta(\nu)\sum
_{k\in
M(\nu)} q_k.
\]
This yields the second claim, \eqref{eq:largenu}.

\begin{appendix}\label{app}
\section*{Appendix: Proof of Lemma~\texorpdfstring{\protect\ref{12345678987456321}}{1.3}}\label{sec-Appendix}

Here we prove Lemma~\ref{12345678987456321}. With the exception of
the positivity of $\nu^*$, this has been proved in \cite{CKMS10}, Theorem
1.5; that proof works under the slightly different assumption
on $v$ that we have here. To obtain the positivity of $\nu^*$, this
proof needs a slight modification, which we briefly indicate now. Fix
$M, N \in\N$. Let $\vect{x}^\ssup{N} = (x_1,\ldots, x_N) \in(\R^d)^N$
be a minimiser of $U_N$ and
$\vect{y}^\ssup{M}=(y_1,\ldots,y_M)$ a minimiser of
$U_M$. Recall that $b$ is the potential range, and let $\delta>0$ be
such that
$v<0$ on $(b-\delta,b)$.
Let $\eps\in(0,\delta/2)$. Let $a \in\R^d$ be such that the shift
$\widetilde{\vect y}^\ssup{M}:= (\widetilde y_1,\ldots,\widetilde
y_M):=( y_1+a,\ldots,y_M+a)$ satisfies:
\begin{itemize}
\item all points from $\widetilde{\vect y}^\ssup{M}$ and $\vect{x}^\ssup{N}$
have distance $|x_i - \widetilde y_j| \geq b - \delta+\eps$ [and
hence $v(|x_i - \widetilde y_j|) \leq0$];
\item there is at least one pair of particles $(x_i,\widetilde y_j)$
with distance $|x_i- \widetilde y_j| \leq b-\eps$.
\end{itemize}
Let $\vect{x}^\ssup{N+M}:=(\vect{x}^\ssup{N},\widetilde{\vect y}^\ssup
{M}) \in(\R^d)^{N+M}$. Let
$c:= -\sup_{r\in[b-\delta+\eps,b-\eps]} v(r)>0$. Then we have
\begin{eqnarray*}
E_{N+M} & \leq U\bigl(\vect{x}^\ssup{N+M}\bigr) \leq U\bigl(
\vect{x}^\ssup{N}\bigr) + U\bigl(\widetilde{\vect y}^\ssup{M}
\bigr) - c = E_N + E_M - c.
\end{eqnarray*}
In particular, the sequences $(E_N)_{N\in\N}$ and $(E_N- c)_{N\in\N}$
are subadditive, whence
\[
e_\infty= \lim_{N\to\infty} \frac{E_N}{N} = \lim
_{N\to\infty} \frac
{E_N- c}{N} = \inf_{N\in\N}
\frac{E_N- c}{N}.
\]
Because of the stability of the pair potential, we have $e_\infty
>-\infty$. The inequality $e_\infty\leq(E_N-c)/N$ for any $N$ leads to
$E_N-N e_\infty\geq c$ for any $N$, and this is the positivity of $\nu^*$.
\end{appendix}

%



\printaddresses

\begin{thebibliography}{21}

\bibitem{AZ98}
%
\begin{bbook}[mr]
\bauthor{\bsnm{Aigner},~\bfnm{Martin}\binits{M.}} \AND
\bauthor{\bsnm{Ziegler},~\bfnm{G{\"u}nter~M.}\binits{G.~M.}}
(\byear{2004}).
\btitle{Proofs from {T}he {B}ook},
\bedition{3rd} ed.
\bpublisher{Springer},
\blocation{Berlin}.
\bid{doi={10.1007/978-3-662-05412-3}, mr={2014872}}
\end{bbook}
%
\bptok{imsref}%
\endbibitem

\bibitem{yeung}
%
\begin{barticle}[mr]
\bauthor{\bsnm{Au Yeung},~\bfnm{Yuen}\binits{Y.}},
\bauthor{\bsnm{Friesecke},~\bfnm{Gero}\binits{G.}} \AND
\bauthor{\bsnm{Schmidt},~\bfnm{Bernd}\binits{B.}}
(\byear{2012}).
\btitle{Minimizing atomic configurations of short range pair potentials
in two dimensions: Crystallization in the {W}ulff shape}.
\bjournal{Calc. Var. Partial Differential Equations}
\bvolume{44}
\bpages{81--100}.
\bid{doi={10.1007/s00526-011-0427-6}, issn={0944-2669}, mr={2898772}}
\end{barticle}
%
\bptok{imsref}%
\endbibitem

\bibitem{CKMS10}
%
\begin{barticle}[mr]
\bauthor{\bsnm{Collevecchio},~\bfnm{Andrea}\binits{A.}},
\bauthor{\bsnm{K{\"o}nig},~\bfnm{Wolfgang}\binits{W.}},
\bauthor{\bsnm{M{\"o}rters},~\bfnm{Peter}\binits{P.}} \AND
\bauthor{\bsnm{Sidorova},~\bfnm{Nadia}\binits{N.}}
(\byear{2010}).
\btitle{Phase transitions for dilute particle systems with
{L}ennard--{J}ones potential}.
\bjournal{Comm. Math. Phys.}
\bvolume{299}
\bpages{603--630}.
\bid{doi={10.1007/s00220-010-1097-5}, issn={0010-3616}, mr={2718925}}
\end{barticle}
%
\bptok{imsref}%
\endbibitem

\bibitem{cly}
%
\begin{barticle}[mr]
\bauthor{\bsnm{Conlon},~\bfnm{Joseph~G.}\binits{J.~G.}},
\bauthor{\bsnm{Lieb},~\bfnm{Elliott~H.}\binits{E.~H.}} \AND
\bauthor{\bsnm{Yau},~\bfnm{Horng-Tzer}\binits{H.-T.}}
(\byear{1989}).
\btitle{The {C}oulomb gas at low temperature and low density}.
\bjournal{Comm. Math. Phys.}
\bvolume{125}
\bpages{153--180}.
\bid{issn={0010-3616}, mr={1017745}}
\end{barticle}
%
\bptok{imsref}%
\endbibitem

\bibitem{dMaso}
%
\begin{bbook}[mr]
\bauthor{\bsnm{Dal Maso},~\bfnm{Gianni}\binits{G.}}
(\byear{1993}).
\btitle{An Introduction to {$\Gamma$}-Convergence}.
\bseries{Progress in Nonlinear Differential Equations and Their Applications}
\bvolume{8}.
\bpublisher{Birkh\"auser},
\blocation{Boston, MA}.
\bid{doi={10.1007/978-1-4612-0327-8}, mr={1201152}}
\end{bbook}
%
\bptok{imsref}%
\endbibitem

\bibitem{DZ98}
%
\begin{bbook}[mr]
\bauthor{\bsnm{Dembo},~\bfnm{Amir}\binits{A.}} \AND
\bauthor{\bsnm{Zeitouni},~\bfnm{Ofer}\binits{O.}}
(\byear{1998}).
\btitle{Large Deviations Techniques and Applications},
\bedition{2nd} ed.
\bseries{Applications of Mathematics (New York)}
\bvolume{38}.
\bpublisher{Springer},
\blocation{New York}.
\bid{doi={10.1007/978-1-4612-5320-4}, mr={1619036}}
\end{bbook}
%
\bptok{imsref}%
\endbibitem


\bibitem{fefferman}
%
\begin{barticle}[mr]
\bauthor{\bsnm{Fefferman},~\bfnm{Charles~L.}\binits{C.~L.}}
(\byear{1985}).
\btitle{The atomic and molecular nature of matter}.
\bjournal{Rev. Mat. Iberoam.}
\bvolume{1}
\bpages{1--44}.
\bid{doi={10.4171/RMI/1}, issn={0213-2230}, mr={0834355}}
\end{barticle}
%
\bptok{imsref}%
\endbibitem

\bibitem{ghm01}
%
\begin{bincollection}[mr]
\bauthor{\bsnm{Georgii},~\bfnm{Hans-Otto}\binits{H.-O.}},
\bauthor{\bsnm{H{\"a}ggstr{\"o}m},~\bfnm{Olle}\binits{O.}} \AND
\bauthor{\bsnm{Maes},~\bfnm{Christian}\binits{C.}}
(\byear{2001}).
\btitle{The random geometry of equilibrium phases}.
In \bbooktitle{Phase Transitions and Critical Phenomena, {V}ol. 18}.
\bpages{1--142}.
\bpublisher{Academic Press},
\blocation{San Diego, CA}.
\bid{doi={10.1016/S1062-7901(01)80008-2}, mr={2014387}}
\end{bincollection}
%
\bptok{imsref}%
\endbibitem

\bibitem{hillbook}
%
\begin{bbook}[mr]
\bauthor{\bsnm{Hill},~\bfnm{Terrell~L.}\binits{T.~L.}}
(\byear{1956}).
\btitle{Statistical Mechanics: {P}rinciples and Selected Applications}.
\bpublisher{McGraw-Hill Book},
\blocation{New York}.
\bid{mr={0120800}}
\end{bbook}
%
\bptok{imsref}%
\endbibitem

\bibitem{hl01}
%
\begin{bbook}[mr]
\bauthor{\bsnm{Hiriart-Urruty},~\bfnm{Jean-Baptiste}\binits{J.-B.}} \AND
\bauthor{\bsnm{Lemar{\'e}chal},~\bfnm{Claude}\binits{C.}}
(\byear{2001}).
\btitle{Fundamentals of Convex Analysis}.
\bpublisher{Springer},
\blocation{Berlin}.
\bid{doi={10.1007/978-3-642-56468-0}, mr={1865628}}
\end{bbook}
%
\bptok{imsref}%
\endbibitem

\bibitem{J}
%
\begin{barticle}[mr]
\bauthor{\bsnm{Jansen},~\bfnm{Sabine}\binits{S.}}
(\byear{2012}).
\btitle{Mayer and virial series at low temperature}.
\bjournal{J. Stat. Phys.}
\bvolume{147}
\bpages{678--706}.
\bid{doi={10.1007/s10955-012-0490-1}, issn={0022-4715}, mr={2930575}}
\bptnote{check year}%
\end{barticle}
%
\bptok{imsref}%
\endbibitem

\bibitem{JK}
%
\begin{barticle}[mr]
\bauthor{\bsnm{Jansen},~\bfnm{Sabine}\binits{S.}} \AND
\bauthor{\bsnm{K{\"o}nig},~\bfnm{Wolfgang}\binits{W.}}
(\byear{2012}).
\btitle{Ideal mixture approximation of cluster size distributions at
low density}.
\bjournal{J. Stat. Phys.}
\bvolume{147}
\bpages{963--980}.
\bid{doi={10.1007/s10955-012-0499-5}, issn={0022-4715}, mr={2946631}}
\bptnote{check year}%
\end{barticle}
%
\bptok{imsref}%
\endbibitem

\bibitem{lebowitz}
%
\begin{barticle}[mr]
\bauthor{\bsnm{Lebowitz},~\bfnm{J.~L.}\binits{J.~L.}} \AND
\bauthor{\bsnm{Penrose},~\bfnm{O.}\binits{O.}}
(\byear{1977}).
\btitle{Cluster and percolation inequalities for lattice systems with
interactions}.
\bjournal{J. Stat. Phys.}
\bvolume{16}
\bpages{321--337}.
\bid{issn={0022-4715}, mr={0676494}}
\end{barticle}
%
\bptok{imsref}%
\endbibitem

\bibitem{murmann}
%
\begin{barticle}[mr]
\bauthor{\bsnm{M{\"u}rmann},~\bfnm{Michael~G.}\binits{M.~G.}}
(\byear{1975}).
\btitle{Equilibrium distributions of physical clusters}.
\bjournal{Comm. Math. Phys.}
\bvolume{45}
\bpages{233--246}.
\bid{issn={0010-3616}, mr={0413957}}
\end{barticle}
%
\bptok{imsref}%
\endbibitem

\bibitem{pechersky}
%
\begin{barticle}[mr]
\bauthor{\bsnm{Pechersky},~\bfnm{E.}\binits{E.}} \AND
\bauthor{\bsnm{Yambartsev},~\bfnm{A.}\binits{A.}}
(\byear{2009}).
\btitle{Percolation properties of the non-ideal gas}.
\bjournal{J.~Stat. Phys.}
\bvolume{137}
\bpages{501--520}.
\bid{doi={10.1007/s10955-009-9856-4}, issn={0022-4715}, mr={2564287}}
\end{barticle}
%
\bptok{imsref}%
\endbibitem

\bibitem{R81}
%
\begin{barticle}[mr]
\bauthor{\bsnm{Radin},~\bfnm{Charles}\binits{C.}}
(\byear{1981}).
\btitle{The ground state for soft disks}.
\bjournal{J. Stat. Phys.}
\bvolume{26}
\bpages{365--373}.
\bid{doi={10.1007/BF01013177}, issn={0022-4715}, mr={0643714}}
\end{barticle}
%
\bptok{imsref}%
\endbibitem

\bibitem{Ru99}
%
\begin{bbook}[mr]
\bauthor{\bsnm{Ruelle},~\bfnm{David}\binits{D.}}
(\byear{1999}).
\btitle{Statistical Mechanics: Rigorous Results}.
\bpublisher{World Scientific},
\blocation{River Edge, NJ}.
\bid{doi={10.1142/4090}, mr={1747792}}
\end{bbook}
%
\bptok{imsref}%
\endbibitem

\bibitem{sator}
%
\begin{barticle}[mr]
\bauthor{\bsnm{Sator},~\bfnm{N.}\binits{N.}}
(\byear{2003}).
\btitle{Clusters in simple fluids}.
\bjournal{Phys. Rep.}
\bvolume{376}
\bpages{1--39}.
\bid{doi={10.1016/S0370-1573(02)00583-5}, issn={0370-1573}, mr={1977747}}
\end{barticle}
%
\bptok{imsref}%
\endbibitem

\bibitem{Th06}
%
\begin{barticle}[mr]
\bauthor{\bsnm{Theil},~\bfnm{Florian}\binits{F.}}
(\byear{2006}).
\btitle{A proof of crystallization in two dimensions}.
\bjournal{Comm. Math. Phys.}
\bvolume{262}
\bpages{209--236}.
\bid{doi={10.1007/s00220-005-1458-7}, issn={0010-3616}, mr={2200888}}
\end{barticle}
%
\bptok{imsref}%
\endbibitem

\bibitem{zessin}
%
\begin{barticle}[mr]
\bauthor{\bsnm{Zessin},~\bfnm{H.}\binits{H.}}
(\byear{2008}).
\btitle{A theorem of {M}ichael {M}\"urmann revisited}.
\bjournal{Izv. Nats. Akad. Nauk Armenii Mat.}
\bvolume{43}
\bpages{69--80}.
\bid{doi={10.3103/s11957-008-1004-y}, issn={0002-3043}, mr={2465000}}
\end{barticle}
%
\bptok{imsref}%
\endbibitem
\end{thebibliography}
\end{document}